\documentclass[10pt]{amsart}
\usepackage{amssymb}

\usepackage{tikz, tikz-cd}
\usetikzlibrary{arrows, intersections, calc, shapes.geometric}
\usetikzlibrary{calc,fit,backgrounds,positioning,arrows,angles, hobby}
\usepackage{tikz-cd} 
\newdimen\XCoord
\newdimen\YCoord
\pgfkeys{/tikz/savenumber/.code 2 args={\global\edef#1{#2}}}

\usepackage[all]{xy}

\newcommand\pDg{Q}
\newcommand\mK{K}
\newcommand\uP{N}
\newcommand\Af{A}
\newcommand\At{B}
\newcommand\Pa{P}

\newcommand\uM{N}
\newcommand\Tf{T}

\newcommand\lra{\longrightarrow}

\newcommand\HC{\operatorname{HC}}

\newcommand\ilog{\operatorname{ilog}}

\newcommand\mrurl[1]{MR. \url{http://www.ams.org/mathscinet-getitem?mr=#1}}
\newcommand\aurl[1]{\url{file:///data/ms/archive/#1}}
\newcommand\durl[1]{doi. \url{http:///dx.doi.org/#1}}
\newcommand\purl[1]{\url{file:///data/ms/print/#1}}
\newcommand\surl[1]{\url{file:///data/ms/screen/#1}}
\newcommand\zurl[1]{Zbl. \url{http://www.emis.de/zmath-item?#1}}
\newcommand\Aurl[1]{Arxiv. \url{http://arxiv.org/abs/#1}}

\newcommand\norm[1]{\|#1\|}

\newcommand\boxb[1]{\square_b}

\newcommand\ff{\operatorname{ff}}

\numberwithin{equation}{section}

\newcommand\paperbody%
        {}
\newcommand\appendixbody%
{\appendix\setcounter{equation}{0}
        }

\newtheorem{lemma}{Lemma}
\newtheorem{proposition}{Proposition}

\newtheorem{non-theorem}{Non-Theorem}
\newtheorem{conjecture}{Conjecture}

\theoremstyle{remark}
\newtheorem{definition}{Definition}
\newtheorem{remark}{Remark}

\newcommand\Vb{\mathcal{V}_{\operatorname{b}}}
\newcommand\bV{\mathcal{V}_{\operatorname{b}}}
\newcommand\cV{\mathcal{V}_{\operatorname{c}}}

\newcommand\bo{\operatorname{b}}

\newcommand\cFTs{{}^{\Phi}\overline{T}\kern-1pt{}^*}

\newcommand\eps{\varepsilon}

\newcommand\Nul{\operatorname{Nul}}

\newcommand\Tr{\operatorname{Tr}}

\newcommand\Vol{\operatorname{Vol}}

\newcommand\SO{\operatorname{SO}}
\newcommand\GL{\operatorname{GL}}

\newcommand\SL{\operatorname{SL}}

\newcommand\cA{\mathcal{A}}

\newcommand\cM{\mathcal M}

\newcommand\cJ{\mathcal{J}}

\renewcommand\cV{\mathcal{V}}

\newcommand\dcF{{}^2\kern-1.5pt\mathcal{F}}
\newcommand\DcA{{}^2\kern-3pt\mathcal{A}}
\newcommand\DcF{{}^2\kern-1.5pt\mathcal{F}}
\newcommand\DcL{{}^2\kern-1.5pt\mathcal{L}}

\newcommand\bbC{\mathbb C}

\newcommand\bbK{\mathbb K}

\newcommand\bbP{\mathbb P}

\newcommand\bbR{\mathbb R}
\newcommand\bbS{\mathbb S}

\newcommand\bbZ{\mathbb Z}

\newcommand\CIc{{\mathcal{C}}^{\infty}_\text{c}}

\newcommand\CI{{\mathcal{C}}^{\infty}}

\newcommand\Diag{\operatorname{Diag}}

\newcommand\Diff[1]{\operatorname{Diff}^{#1}}

\newcommand\Diffb[1]{\operatorname{Diff}^{#1}_{\text{b}}}

\newcommand\cFNs{{}^{\Phi}\overline N\kern-1pt{}^*}

\newcommand\Hom{\operatorname{Hom}}

\newcommand\Id{\operatorname{Id}}

\newcommand\SU{\operatorname{SU}}

\newcommand\ci{${\mathcal{C}}^\infty$}

\newcommand\dCI{\dot{\mathcal{C}}^{\infty}}

\newcommand\ha{\frac{1}{2}}

\newcommand\pa{\partial}

\renewcommand\Re{\operatorname{Re}}
\renewcommand\Im{\operatorname{Im}}

\newcommand\Mas{\text{ as }}
\newcommand\Mat{\text{ at }}

\newcommand\Mfor{\text{ for }}

\newcommand\Min{\text{ in }}

\newcommand\Mon{\text{ on }}

\newcommand\Mover{\text{ over }}

\newcommand\Mthen{\text{ then }}

\newcommand\ifitem[2]{\def\test{#2}\ifx\test\@empty\else\item #1#2\fi}
\begin{document}
\title[CoSL2]
{Compactification of $\SL(2)$}

\author{Pierre Albin}
\author{Panagiotis Dimakis}
\author{Richard Melrose}
\address{Department of Mathematics, University of Illinois at Urbana-Champaign}
\email{palbin@illinois.edu}
\address{Department of Mathematics, Massachusetts Institute of Technology}
\email{pdimakis@mit.edu}
\address{Department of Mathematics, Massachusetts Institute of Technology}
\email{rbm@math.mit.edu}
%
\begin{abstract} We discuss `hd-compactifications' of $\SL(2,\bbK)$ for
  $\bbK=\bbC$ or $\bbR.$ These are compact manifolds with boundary on which
  both the Schwartz and the Harish-Chandra Schwartz spaces are shown to be
  relatively standard spaces of conormal functions relative to the
  boundary. Closure under convolution and other module properties are shown
  to follow from the structure of appropriate generalized product spaces
  and the functorial properties of conormal functions and smooth maps
  between manifolds with corners. It is anticipated that a similar approach
  applies to general real reductive Lie groups, with the additional
  complications for $\SL(n,\bbK)$ being essentially combinatorial.
\end{abstract}
\maketitle

\tableofcontents

\thanks{P.A. was supported by NSF grant DMS-1711325.}

\section*{Introduction}

In this note we discuss real compactifications of the groups $\SL(2,\bbR)$
and $\SL(2,\bbC)$ and associated spaces. These are special cases of the
`hd-compactification' which will be described elsewhere for $\SL(n,\bbK)$
and $\GL(n,\bbK).$ We conjecture that such a compactification exists, and
in an appropriate sense is unique, for any real reductive Lie group. We
view the hd-compactification as the real analogue of the `wonderful
compactification' of de Concini and Procesi, to which it is closely
related. Similar compactifications have been considered elsewhere, in
particular by Mazzeo and Vasy \cite{MR2175410}, especially for homogeneous
spaces.

We hope that the approach to the subject of analysis on groups presented
here may have more substantial consequences, in this note we restrict
attention to relatively well-know results approached with these less
familiar techniques, which have their origins in scattering theory and
geometric analysis of non-compact and singular spaces. Perhaps most 
relevant is the study of edge vector fields and the corresponding 
geometric constructions. In fact we encounter here not only the edge
calculus of Mazzeo, \cite{MR1133743}, which has its origins in the study of
hyperbolic space in \cite{MR89c:58133}, but also the b-calculus \cite{MR1348401}.

\begin{definition}\label{CoSL2.239} By an hd-compactification of a Lie group 
$G$ we mean a compactification, in principle to a compact manifold with corners, 
$G[1],$ with three properties:-
\begin{itemize}
\item Inversion extends to a diffeomorphism of $G[1].$
\item The right-invariant vector fields span the `iterated edge' vector
  fields associated to an iterated boundary fibration structure of $G[1].$
\item Together the left- and right-invariant vector fields span the Lie
  algebra, $\Vb(G[1])$ of tangent vector fields. 
\end{itemize}
\end{definition}

As is indicated below, these properties imply that the Schwartz space is
identified with the space of smooth functions on the compactification,
vanishing to infinite order at the boundary. More significantly
Harish-Chandra's Schwartz space, denoted here $\HC(G),$ is identified with
the space of conormal functions with respect to the boundary which are
log-rapidly decaying relative to a fixed power weight determined by the
Haar measure of the group (and encoded by Harish-Chandra in the decay
properties of a spherical function). This corresponds to the smallest power
of a boundary defining function (in general a product of powers of defining
functions) not in $L^2.$ For the convenience of the reader an appendix on
conormal functions is included.

For $\SL(2,\bbR)$ the hd-compactification is a solid 3-torus, i.e.\ is
diffeomorphic to the product of the circle, $\SO(2),$ and a closed 2-disc
which can be identified with the closure of the positive symmetric $2\times
2$ matrices of trace $1.$ This in turn is a radial compactification of the
space of positive matrices of determinant one. In general, for
$\SL(n,\bbC)$ or $\SL(n,\bbR),$ it is necessary to desingularize the
stratified space given as the closure of the positive hermitian or
symmetric matrices of trace $1,$ hence the designation `hd'. This
hd-compactification is shown to be closely related to (and at least for
$\SL(n,\bbC)$ derivable from) the wonderful compactification of de Concini
and Procesi, \cite{MR718125}. As pointed out to us by Eckhard Meinrenken,
the compactification of $\SL(2,\bbR)$ can be obtained as the closure of the
image of the radial projection into the sphere in the $2\times 2$ matrices,
i.e.\ as the closure in the sphere of the matrices with positive
determinant..

One of the fundamental properties of the Harish-Chandra Schwartz space is
that it is closed under convolution. A geometric proof of this is given
here by defining an associated compactification, $G[2],$ of $G^2$ with the
property, amongst others, that multiplication $(g,h)\longmapsto gh^{-1}$
extends to a smooth map $G[2]\longrightarrow G[1].$ Closure of $\HC(G)$
under convolution then follows from push-forward/pull-back properties of
conormal functions under b-fibrations, of which this map is an example. The
space $G[2]$ is the `double' space for the (in general iterated) edge
structure.

We give a second (larger) compactification of $\SL(2,\bbR)$ relative to a
parabolic subgroup and use it to recover the result that the Harish-Chandra
functions on the quotient by the associated unipotent group form a
convolution module over the Harish-Chandra space of the group. Use of this
space also serves to show that the spherical function, $\Phi,$ is log-smooth.

As an elementary illustration of our approach, consider the 1-dimensional
multiplicative group, which appears below as the positive diagonal subgroup
of $\SL(2).$ We radially compactify $\GL_+=\bbR_+$ to a closed interval,
for instance using the diffeomorphism to the interior
\begin{equation}
\GL_+\ni\tau\longrightarrow \frac2\pi\arctan \tau\in[0,1]=\GL_+[1].
\label{14.11.2018.5}\end{equation}
Thus $\frac1\tau$ is a defining function near the top boundary, and $\tau$
itself defines the lower boundary. Rather trivially, this is the unique
hd-compacification, with the Lie algebra, spanned by $\tau\pa_{\tau}$
generating the b-vector fields. The Harish-Chandra space is  
\begin{equation}
\HC(\GL_+)=(\log\rho)^{\infty}\cA(\GL_+[1])
\label{SL2.8}\end{equation}
is the space of conormal functions (having stable regularity under
application of b-differential operators) which decay faster at both boundaries
than any inverse power of the logarithm of a defining function.

As noted above, the twisted product $\chi:GL_+^2\ni (\sigma ,\tau
)\longmapsto \sigma /\tau $ extends smoothly to a b-fibration
\begin{equation}
\chi:\GL_+[2]=[\GL_+^2;\{(0,0)\},\{1,1)\}]\longrightarrow \GL_+[1].
\label{14.11.2018.6}\end{equation}
In this case the double space, defined by blowing up the corner, is the
usual product space for the b-calculus.  The two stretched projections
$\pi_R,$ $\pi_L:\GL_+[2]\longrightarrow \GL_+[1]$ are also b-fibrations. The
convolution product is captured by the diagram and formula
\begin{equation}
\begin{gathered}
\xymatrix{
G[1]\\
G[2]\ar[u]^{\pi_L}\ar[r]_{\pi_R}\ar[d]_{\chi}&G[1]\\
G[1]
}\\
f_1*f_2=(\pi_{L})_*\left(\chi^*f_1\pi_R^*f_2dg_R\right);
\end{gathered}
\label{CoSL2.275}\end{equation}
which carries over to the general case; see also Figure~\ref{GL2}.
\begin{figure}\label{GL2}
\begin{tikzpicture}[scale=1.4,>=stealth]
\node at (1,1) {$\GL_+[2]$};
\draw (0,.3) -- (0,2);
\draw (.3,0) --  (2,0);
\draw (0,.3) arc [radius=.3, start angle = 90, end angle = 0] ;
\draw (0,2) -- (1.7,2);
\draw (2,0) -- (2,1.7);
\draw (1.7,2) arc [radius=.3, start angle = 180, end angle = 270] ;
\node at (2.75,1) [above] {$\pi_R$} ;
\node at (3.95,1) {$\GL_+[1]$} ;
\draw [->] (2.5,1) -- (3,1);
\draw (3.5,0) -- (3.5,2) ;
\node [green] at (3.5,1.7) {$\bullet$} ;
\node [green] at (3.5,.3) {$\bullet$} ;
\node [blue] at (1.7,3.5) {$\bullet$} ;
\node [blue] at (.3,3.5) {$\bullet$} ;
\node [red] at (-1.3,-.9) {$\bullet$} ;
\node [red] at (-.9,-1.3) {$\bullet$} ;
\draw [->] (1,2.5) -- (1,3);
\node at (1,2.75) [left] {$\pi_L$} ;
\node at (1,3.75) {$\GL_+[1]$} ;
\draw (0,3.5) -- (2,3.5) ;
\draw [->] (-.2,-.2) -- (-.7,-.7);
\node at (-.45,-.45) [above left] {$\chi$} ;
\draw (-1.6,-.6) -- (-.6,-1.6) ;
\node at (-1.5,-1.5) {$\GL_+[1]$} ;
\draw [rounded corners, green]
(0,1.85) -- (1.5,1.83) -- (1.60,1.80) -- (1.66,1.765) -- (1.69,1.74) -- (1.73,1.71) -- (1.8,1.67) -- (2,1.62) ;
\draw [rounded corners, green]
(2,0.15) -- (0.5,0.17) -- (0.40,0.2) -- (0.34,0.235) -- (0.31,0.26) -- (0.27,0.29) -- (0.2,0.33) -- (0,0.38) ;
\draw [rounded corners, blue]
(0.15,2) -- (0.17,0.5) -- (0.2,0.40) -- (0.235,0.34) -- (0.26,0.31) -- (0.29,0.27) -- (0.33,0.2) -- (0.38,0) ;
\draw [rounded corners, blue]
(1.85,0) -- (1.83,1.5) -- (1.80,1.60) -- (1.765,1.66) -- (1.74,1.69) -- (1.71,1.73) -- (1.67,1.8) -- (1.62,2) ;
\draw [rounded corners, red] (0.1,0.2828) -- (0.13, .7) -- (0.17, 1) --
(.27,1.4) -- (0.4,1.6) -- (.6,1.73) -- (1,1.83) -- (1.3, 1.87) -- (1.7172,1.9);
\draw [rounded corners, red] (0.2828,0.1) -- (.7, .13) -- (1,.17) -- (1.4,.27)
-- (1.6,0.4) -- (1.73,.6) -- (1.83,1) -- (1.87,1.3) -- (1.9,1.7172);
\end{tikzpicture}
\caption{The b-product of intervals}
\end{figure}
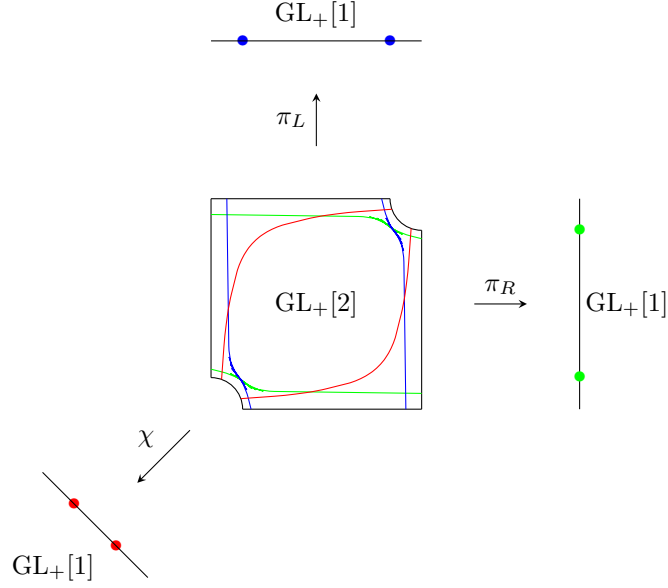

In \S\ref{S.Com} the notion of hd-compactification is discussed and the
compactification of $\SL(2)$ is described. In \S\ref{S.HC} the
Harish-Chandra and Schwartz spaces are identified and in \S\ref{S.Con} they
are shown to be closed under convolution. The compactification $G[1;N]$
corresponding to the action of unipotent subgroup is presented in
\S\ref{S.Iwa} and used in \S\ref{S.mod} to show that the space $\HC(G/N)$
is a module over $\HC(G).$ Following Crisp and Higson, \cite{MR3666051},
the relevance of these constructions for induced representations is
recalled in \S\ref{S.ind} and some properties of the intertwining operators
for a parabolic group and its opposite are given in \S\ref{S.int}.

The results concerning the Harish-Chandra Schwartz space given here are
largely in reply to questions raised by Nigel Higson; we hope to answer
more of these in due course. The authors also thank Roman Bezrukavnikov and
Eckhard Meinrenken for helpful discussions.

\paperbody
\section{Compactification}\label{S.Com}
By a compactification of a non-compact manifold without boundary $M$ we
mean a compact manifold, necessarily with boundary and generally with
corners, $M[1],$ and a diffeomorphism
\begin{equation*}
	\xymatrix{
	M \ar@{^(->}[r]^-{i} & M[1] }
\end{equation*}
onto the interior of $M[1].$ The `1' here corresponds to a compactification of
$M,$ $M[2]$ to a compactification of $M^2,$ etc.  Often there is more than
one compactification of interest for a given manifold. In that case we add
a distinguishing modifier, e.g., $M[2;\bo].$

Two compactifications are \underline{equivalent} if there is a
diffeomorphism between them intertwining the injection diffeomorphisms
\begin{equation*}
	\xymatrix{
	& {M[1,i]} \ar@{<->}[dd] \\
	M \ar@{^(->}[ru]^-{i} \ar@{^(->}[rd]^-{i'} & \\
	& {M[1,i']}. }
\end{equation*}

Compactifications often arise through the resolution of a singular space
and result in manifolds with corners with iterated boundary fibrations.  As
in \cite{1807.08299} we call these simply `iterated spaces'. Such an
iterated space is a compact manifold with corners with compatible
fibrations at each of its boundary hypersurfaces. Since these do not arise
in any significant way in the present discussion of $\SL(2,\bbR)$ and
$\SL(2,\bbC)$ (denoted collectively $\SL(2)$) we do not recall the notion
-- but it is needed for $n>2.$ For $\SL(2)$ the compactifications of the
groups are manifolds with boundary and correspondingly the iterated structure
is an \emph{edge structure}. This corresponds to a fibration of the
boundary, i.e.\ it is the total space of a smooth fibre bundle with compact
fibres
\begin{equation}
\xymatrix{
F\ar@{-}[r]&\pa M[1]\ar[d]\\
&B.
}
\label{CoCLG.99}\end{equation}
The edge vector fields associated to this structure are 
\begin{equation}
\cV_{\text{e}}(M[1])=\{V\in\bV(M[1]);V\text{ is tangent to the fibres \eqref{CoCLG.99}}\}.
\end{equation}
Here $\bV(M[1])$ is the Lie algebra of all smooth vector fields tangent to
  the boundary (or, for a manifold with corners, tangent to all boundary faces).

\begin{definition}\label{CoSL2.103} 
By an hd-compactfication of a Lie group we mean a compactification $G 
\hookrightarrow G[1]$ in the sense discussed above, with the additional
properties:
\begin{itemize}
\item [(a)] Inversion extends to a diffeomorphism of $G[1],$
\item [(b)] The right-invariant vector fields on $G$ lift to be smooth on
  $G[1],$ tangent to all boundary faces, and to span (over $\CI(G[1])$) the
  Lie algebra of iterated edge vector fields -- those vector fields tangent
  to an appropriate iterated structure on $\pa G[1],$
\item [(c)] The span of left and right invariant vector fields on $G[1]$ is the Lie algebra of
all vector fields tangent to the boundary, $\cV_b(G[1]).$
\end{itemize}
\end{definition}

Observe that (a) and (b) imply that the left invariant vector fields span
the iterated structure which is the image of the one in (b) under
inversion. Thus (c) is by way of a transversality condition for these two
fibrations. Similarly it follows from (b) that the left and right actions
of $G$ on itself extend to actions on $G[1].$

\begin{conjecture}\label{CoSL2.101} 
Any real reductive Lie group has a unique hd-compactification up to equivalence.
\end{conjecture}

Although not discussed here the construction of $G[1]$ and $G[2]$ below can
be extended to higher products giving a `generalized product' $G[*]$ which
is a simplicial space with additional functorial properties.

To compactify $\SL(n,\bbK)$ we start from the right polar decomposition
\begin{equation}\label{9.11.1}
	G=\mK\Af, \ g=ka,\ a = (g^*g)^{1/2}.
\end{equation}
Thus $\mK$ is the maximal compact subgroup. For $n=2,$ $\mK=SO(2)\subseteq
SU(2)$ in the real and complex cases and and $\Af$ is the space of positive
definite Hermitian $2\times2$ matrices of determinant one in the complex case and the
real subspace for $\SL(2,\bbR).$

Let $\At$ be the corresponding space, of Hermitian or symmetric matrices
respectively, that are positive definite and of trace $1.$ Since $a$ in
\eqref{9.11.1} has determinant equal to 1 the map
\begin{equation*}
	\Af \ni a \mapsto b = (\Tr(a))^{-1}a \in \At,\ \Tr(a)=(\det(b))^{-\frac12},
\end{equation*}
is a diffeomorphism.

Let $\At[1]$ be the closure of $\At$ in the $2\times 2$ matrices - the space of
non-negative Hermitian or real symmetric matrices of trace $1.$ In the general case,
$n>2,$ the closure is not smooth and $\At[1]$ is defined as a resolution of the
resulting stratified space. Here however,

\begin{lemma} $\At[1]$ is a closed ball of dimension $3$ for $\SL(2,\bbC)$
  and dimension $2$ for $\SL(2,\bbR).$
\end{lemma}

\begin{proof}
The elements of $\At[1]\setminus B$ are non-negative Hermitian or
real-symmetric $2\times2$ matrices of rank $1$ and trace $1.$ Thus $1$ is
an eigenvalue.  In a small neighborhood of $\pa\At[1]$ in $\At[1]$
there is necessarily an eigenvalue close to one and another close to zero,
$1-s$ and $s,$ respectively.  Since the multiplicity of the eigenvalues is
constant the eigenspaces are smooth and the eigen-decomposition allows the
neighbourhood to be identified with, in the complex case,
\begin{equation*}
	\bbP \times [0,\eps)_s \mapsto 
	(1-s)q(\xi) + sq^{\perp}(\xi), \quad
	\xi \in \bbP.
\end{equation*}
Here $q$ is orthogonal projection onto the line in $\bbC^2$ determined by
$\xi\in\bbP=\bbS^2.$  In the real case the eigenvectors are necessarily
real so the identification becomes 
\begin{equation*}
	\bbS \times [0,\epsilon)_s \mapsto (1-s)q(\theta) + sq^{\perp}(\theta).
\end{equation*}
Now $q(\theta)$ is projection onto $\cos\theta e_1+\sin\theta e_2,$
$\theta\in[0,\pi).$ Thus it follows that $\At[1]$ is a compact ball, or disk, with
$s,$ defined near the boundary, as a boundary defining function.
\end{proof}

In fact it is generally more convenient to take a slightly different representation
of a neighbourhood of the boundary of $\At[1].$ Since the eigenvalue near $1$
is smooth near the boundary we may divide by it and consider instead the
space of Hermitian/symmetric $2\times 2$ non-negative matrices with $1$ as
an eigenvalue and with the other eigenvalue suitably small. This gives the
representation
\begin{equation*}
\beta =	q(\theta) + tq^{\perp}(\theta), \quad t\in [0,\eps).
\end{equation*}
Then the corresponding element of $\Af$ is
\begin{equation*}
	a = t^{-\ha}\beta =t^{-1/2}q(\theta) + t^{1/2} q^{\perp}(\theta).
\end{equation*}

\begin{proposition}
For $G = \SL(2,\bbR)$ and $\SL(2,\bbC)$
\begin{equation}
	G[1] = K \times \At[1] 
\label{CoSL2.102}\end{equation}
is an hd-compactification in which the fibration of the boundary
corresponding to the right-invariant vector fields is given by the restriction of
the product decomposition \eqref{CoSL2.102} to $\pa G[1]=K\times\pa \At[1].$
\end{proposition}

\begin{proof}

First note that the adjoint action of $K$ on $\Af$ extends to a smooth
action on $\At[1].$ Indeed under conjugation $g\longmapsto kgk^{-1}$ both
the determinant and the trace are invariant so this action projects to $B$
to the conjugation action there and so extends smoothly to the closure
$\At[1].$ On the boundary $\bbS^2,$ respectively $\bbS,$ the action of
$\SU(2)$ projects to the rotation action of $\SO(3)=\SU(2)/\bbZ_2,$
respectively $\SO(2)$ projects to the rotation action of $\SO(2)/\bbZ_2.$
It follows that this action is smooth (and transitive) on the boundary of $\At[1],$ so
the compactification obtained by taking the opposite polar decomposition is
equivalent to \eqref{CoSL2.102}.

We next compute the span of the right-invariant vector fields, of course
this is locally all vector fields on $G,$ so we are only interested in the
behaviour near the boundary. If $g=ka$ and $u\in\mathfrak{g}$ is an element
of the Lie algebra then
\begin{equation}
\exp(su)ka=k(s)a(s)
\label{13.9.2018.1}\end{equation}
is the integral curve of a general right-invariant vector field near $g.$ If  
\begin{equation}
\mathfrak{g}=\mathfrak{k}+\mathfrak{a}
\label{13.9.2018.2}\end{equation}
and $u\in\mathfrak{k}$ then $\exp(su)\in K$ and it follows that the
right-invariant vector fields on $K$ lift smoothly to $G[1]=K\times \At[1]$
and, by transitivity, span all vector fields on $K.$

Thus it suffices to consider \eqref{13.9.2018.1} for $u\in\mathfrak{a}.$ Then
the polar decomposition gives 
\begin{equation}
a(s)^2=g^*g=a\exp(2su)a,\ a=a(0).
\label{13.9.2018.3}\end{equation}
Mapping $\Af$ into $\At$ gives the curve
$b(s)=(\Tr(a(s))^{-1}a(s)=\det(b(s))^{\ha}a(s)$ defined by 
\begin{equation}
b(s)^2=\det(b(s))a\exp(2su)a.
\label{13.9.2018.7}\end{equation}

If $a\in\Af$ approaches the boundary of $G[1]$ along the curve of diagonal
matrices, 
\begin{equation}
a=\begin{pmatrix}t^{-\ha}&0\\0&t^{\ha}
\end{pmatrix}\Mas t\downarrow0
\label{13.9.2018.5}\end{equation}
then
\begin{equation}
u=\begin{pmatrix}1&0\\0&-1
\end{pmatrix}\Longrightarrow a(s)=
\begin{pmatrix}t^{-\ha}e^{s}&0\\0&t^{\ha}e^{-s}
\end{pmatrix},\
\beta(s)=\begin{pmatrix}1&0\\0&te^{-2s}
\end{pmatrix}.
\label{13.9.2018.6}\end{equation}
Similarly
\begin{equation}
u=\begin{pmatrix}0&1\\ 1&0
\end{pmatrix}\Longrightarrow a(s)^2=\begin{pmatrix}t^{-1}\cosh (2s)&\sinh(2s)\\
\sinh(2s)&t\cosh(2s)
\end{pmatrix}.
\label{13.9.2018.9}\end{equation}
The large eigenvalue $\lambda$ of $a(s)^2$ satisfies
\begin{multline}
(\cosh(2s)-t\lambda )(t^2\cosh(2s)-t\lambda )=t^2\sinh^2(2s)\Longrightarrow\\
\lambda =t^{-1}\cosh(2s)(1+t^2s^2F(t^2,s^2))
\label{13.9.2018.10}\end{multline}
where the implicit function theorem shows that $F$ is smooth near $0$ and
$F(0,0)\not=0.$ The corresponding eigenspace is spanned by
\begin{equation}
e_1-tsL(t^2,s^2)e_2
\label{13.9.2018.11}\end{equation}
with $L$ smooth and $L(0)F(0)=1.$

From \eqref{13.9.2018.6} it follows that the corresponding right-invariant
vector field, projected to $\At[1],$ is $-2t\pa_t.$ Similarly for
\eqref{13.9.2018.9} the vector field vanishes with $t$ but with a
coefficient which is a non-vanishing vector field on $\pa \At[1].$ Taking
into account the conjugation action discussed above, this identifies the
span of the right-invariant vector fields with the edge vector fields for
$\mK\times \pa\At[1]\longrightarrow \pa\At[1].$

The inverse of $kuau^{-1},$ where $a$ is positive and diagonal, is
$ua^{-1}u^{-1}k^{-1}.$ Inversion of diagonal matrices in $\Af$ clearly extends
smoothly to $\At[1]$ so it follows that inversion on $G$ extends smoothly to $G[1].$

It is noted above that the radial vector field on $\At$ is, near the
boundary, in the span of the right-invariant vector fields. Since the
conjugation action is in the span of the left and right vector fields and
acts transitively on $\pa \At[1]$ it follows that all tangent vector fields
on $G[1]$ are in the smooth span of the left- and right-invariant vector
fields. Thus $G[1]$ is an hd-compactification.
\end{proof}

\begin{remark}\label{CoSL2.104} An equivalent compactification of $\Af$ can be
  obtained by projecting to the trace-free Hermitian matrices 
\begin{equation}
\Af\ni a\longmapsto a-\ha\Tr(a)\Id\in\Tf.
\label{CoSL2.105}\end{equation}
However it corresponds to the quadratic compactification, rather than the
usual radial compactification, in which $(\Tr(\alpha ^2))^{-1}$ is introduced as
a defining function near infinity. This is well-defined for a linear space
(i.e.\ linear transformations lift to be smooth) but not for an affine space.
\end{remark}

\begin{remark}\label{CoSL2.106} The stabilizer of a fibre $K\times\{q\},$ $q\in\pa\At[1],$
  of the right fibration of the boundary under the right action of $G$ is the parabolic
  subgroup 
\begin{equation}
\Pa(q^\perp)=\{g\in G;gq^{\perp}=cq^{\perp}\}
\label{CoSL2.107}\end{equation}
as is discussed further below. Thus the compactification amounts to the simultaneous
addition of the homogeneous spaces $G/\Pa$ for all parabolic subgroups of $G.$

Geometrically, the resolution $\At[1]$ in the case of $\SL(n)$ has a strong
iterative property. Namely the boundary hypersurfaces are labelled by a
`depth' index which can be taken to be the corank of a limiting matrix. The
corresponding boundary hypersurface of $\At[1]$ is a bundle over the
Grassmannian corresponding to the limiting rank, with fibre the product of
two versions of $\At[1],$ one for the point in the Grassmannian and another
for its orthogonal.  \end{remark}

There is a close relationship between the hd-compactification and the wonderful
compactification of de Concini and Procesi.

\begin{proposition} The adjoint group $\SL(n, \bbC)/\bbZ_n$ has the same positive part as
$\SL(n,\bbC)$ and the closure of the image of $\Af$ in the wonderful compactification of
$\SL(n,\bbC)/\bbZ_n$ is diffeomorphic to $\At[1].$
\end{proposition}

This is quite elementary for $n=2.$ For $\SL(n,\bbC)/\bbZ_n$ the
hd-compactification is given by the real blow-up of the divisors in the
wonderful compactification -- similar constructions occur in
\cite{MetResMod} for the Deligne-Mumford compactification of the Riemann
moduli space.

\section{Schwartz spaces}\label{S.HC}
It follows from Proposition~\ref{CoSL2.102} that for $\SL(2)$ Haar measure
is an edge density on $G[1]$ -- if $t$ is a boundary defining function then
\begin{equation}
dg=t^{-1-2\kappa }\nu =t^{-2\kappa }\nu _{\bo}\text{ near }\pa G[1]
\label{CoSL2.109}\end{equation}
where $\nu$ is a smooth, strictly positive, measure and $\nu _{\bo}$ is a
b-measure, so near the boundary is of the form 
\begin{equation}
\nu _{\bo}=\frac{dt}t\nu_{\pa}
\label{CoSL2.110}\end{equation}
with $\nu_{\pa}$  a positive smooth measure on the boundary.  The weight
$2\kappa$ is the codimension of the fibres over the boundary, i.e.\
\begin{equation}
\kappa =\begin{cases}
\ha&\Mfor\SL(2,\bbR)\\
1&\Mfor\SL(2,\bbC).
\end{cases}
\label{CoSL2.111}\end{equation}
Thus if $L^2_g(G)$ and $L^2_{\bo}(G[1])$ are the $L^2$ spaces, computed
relative to Haar measure and a b-measure respectively, then 
\begin{equation}
L^2_g(G)=\rho ^\kappa L^2_{\bo}(G[1])
\label{CoSL2.112}\end{equation}
where $\rho$ is any boundary defining function.
 
One indication of the relevance of the hd-compactification is that the
extended Schwartz and Harish-Chandra spaces are readily characterized in
terms of $G[1].$ The definition and some of the properties of the spaces of
conormal functions bounded with respect to a weight are recalled in the
appendix.

The space of bounded conormal functions is defined by
\begin{equation*}
	u \in \cA(X)
	\iff
	\sup|Du|<\infty\text{ for all } D \in \Diffb*(X)
\end{equation*}
where $\Diffb*(X)$ is the enveloping algebra of $\bV(X).$ As follows
directly, $\cA(X)$ is a Fr\'echet algebra containing the space, $\CI(X),$
of functions smooth up to the boundary. If $w<0$ is a weight in
the sense of \eqref{CoSL2.147} then
\begin{equation*}
	u \in w\cA(X) \iff
	u/w \in \cA(X) \iff
	\sup|(Du)/w| <\infty \text{ for all } D \in \Diffb*(X).
\end{equation*}

The most obvious weights on a compact manifold with corners are the
products of real powers of defining functions for the various boundary
hypersurface. Here logarithmic weights are also important so, always
choosing a boundary defining function with $\rho <1,$ we define
\begin{equation*}
	\ilog\rho = \frac1{\log{\frac1\rho }} \in \cA(X).
\end{equation*}
Indeed, if $V\in\bV(X)$ is a vector field tangent to the boundary then
$(V\rho )/\rho\in\CI(X)$ and 
\begin{equation}
V\ilog\rho =(\ilog\rho )^2\frac{V\rho }\rho 
\label{CoSL2.108}\end{equation}
so it follows that $\ilog\rho$ is a weight, vanishing at the boundary. We
use the formal notation of $w^\infty w'$ for a weight $w,$ required to vanish at the
boundary and a second weight $w',$ to denote the intersections of the weighted spaces
\begin{equation}
w^\infty w'\cA(X)=\bigcap_{m\in\bbR} w^mw'\cA(X); 
\label{CoSL2.172}\end{equation}
these are again Fr\'echet spaces, now with $\CIc(X\setminus\pa X)$ a dense subspace.

\begin{proposition} For an hd-compactification of a semisimple Lie group
  the Schwartz space is $\dCI(G[1]),$ the space of smooth functions vanishing
  to infinite order at all boundary faces, and the Harish-Chandra
  (Schwartz) space is
\begin{equation}
	\HC(G) = (\ilog\rho)^\infty \rho ^\kappa\cA(G[1]),
\label{9.11.3}\end{equation}
the space of conormal functions with `log-rapid vanishing' at the
boundary relative to the weighted space $\rho ^\kappa\cA(G[1]).$
\end{proposition}

\begin{proof}
Here we consider only $G=\SL(2,\bbK)$ but in fact the proof persists for $\SL(n,\bbK).$

The definition given by Knapp, \cite{MR1880691}, and Wallach, \cite{MR929683},
amounts to the condition
\begin{equation}\label{9.11.2}
	u \in \HC(G)
	\iff
	\frac{\norm{g}^p}{\Xi}D_1D_2u \in L^\infty(G), \quad
	\text{ for all } p, D_1, D_2
\end{equation}
where $D_1$ and $D_2$ are in the left and right enveloping algebras. The
weight $\norm{g}\sim 1/\ilog\rho $ and the spherical
function $\Xi$ is `almost in $L^2$' -- in this case
\begin{equation*}
	ct^{\kappa}\le \Xi \le Ct^{\kappa}\log 1/t \text{ near } \pa
        G[1],\ c,\ C>0.
\end{equation*}

For $\SL(2,\bbR),$ this follows from the fact that $\kappa$ is a double
root of the indicial polynomial of the radial part of the Laplacian on
$\Af$ shifted corresponding to the bottom of the continuous spectrum but
can also be extracted from results in Varadarajan's book
\cite{MR1725738}. A direct proof by push-forward is given below in
Lemma~\ref{CoSL2.234}. Then \eqref{9.11.2} is equivalent to
\eqref{9.11.3}. Note that we have used the consequence of the
hd-compactification conditions that
\begin{equation*}
	D_1D_2 \in \Diffb*(G[1])
\end{equation*}
and these products of left and right invariant operators span the b-differential operators.
\end{proof}

If one thinks in terms of standard harmonic analysis then \eqref{9.11.3} is
equivalent to a statement on the Mellin transform near the boundary. Namely
the Mellin transform -- the Fourier transform in terms of $\log t$ -- is,
in an appropriate normalization, holomorphic in the dual half-space $\Im
s>\kappa$ and uniformly a Schwartz function of $\Re s$ up to, and on, the
limiting line with values in $\CI(\pa G[1]).$ Thus, in terms of the
variable $x=\ilog t,$ these are smooth functions in the usual sense,
vanishing rapidly as $x\downarrow0$ but with a factor of $\exp(-\kappa/x).$
  
\section{Convolution}\label{S.Con}
The left action of $G$ on $G$ is given by integration of the image of the
Lie algebra and since these vector fields extend smoothly to $G[1],$ where
they are complete, the left action extends smoothly. Similarly for the
right action:-
\begin{equation}
G\times G[1]\longrightarrow G[1],\ G[1]\times G\longrightarrow G[1].
\label{CoSL2.113}\end{equation}
However the product itself does not extend to a smooth map from $G[1]^2.$

To resolve this issue we consider an appropriate compactification of $G^2$
obtained by blow-up from $G[1]^2.$ For $\SL(2)$ we take $G[2]=G[2,R]$ to be the
edge compactification of $G[1]^2$ with respect to the right
fibration. More explicitly,
\begin{equation}
	G[2]=G[2,R] =
        K^2\times\At[2,0],\ \At[2,0]=[\At[2,\bo];\beta_{\bo}^{-1}(\pa\Diag)].
\label{CoSL2.117}\end{equation}
Here $\At[2,\bo]$ is the `b-resolution' of $\At[1]^2,$ obtained by blowing
up the corner:- 
\begin{equation}
\At[2,\bo]=[\At[1]^2;(\pa\At[1])^2],\ \beta _{\bo}:\At[2,\bo]\longrightarrow \At[1]^2
\label{CoSL2.116}\end{equation}
being the blow-down map. The subsequent blow up in \eqref{CoSL2.117} is of
the preimage of the diagonal in the boundary. In terms of the product with
$K^2$ on the left, the second blow-up corresponds to the preimage of the
fibre diagonal of the boundary. Neither blow-up affects the interior which
remains $K^2\times\At^2.$ The inclusion of $G^2$ is through the `right'
product inclusion in $(K \times \At[1]) \times (K \times \At[1]).$

Note that the first blow-up is not really necessary to resolve the edge
structure of the manifold. However it seems that this larger resolution (the
blow-ups can be performed in either order) is the more appropriate one here.

\begin{proposition}
The twisted product map
\begin{equation}
	\chi:G\times G \ni (g,h) \mapsto gh^{-1} \in G 
\label{CoSL2.173}\end{equation}
and the two projections lift to b-fibrations 
\begin{equation}
\xymatrix{
G[1]\\
G[2]\ar[r]^{\pi_R}\ar[u]^{\pi_L}\ar[d]_\chi& G[1]\\
G[1].
}
\label{CoSL2.125}\end{equation}
\end{proposition}
\noindent Although there is a corresponding, but different,
left compactification, $G[2,L],$ of $G^2$ we will denote this right
compactification by $G[2].$

\begin{proof} We give a computational proof, although this follows more
  abstractly from the properties of the hd-compactification. As noted
  above, the product map extends to $G\times G[1]$ and $G[1]\times G$ and,
  since the blow-ups in \eqref{CoSL2.116} and \eqref{CoSL2.117} are in the
  factors of $\At[1]$ and the conjugation action of $K$ on $B[1]$ is
  smooth, it suffices to consider the behaviour of the product of two
  elements of $\At[1]$ near the boundary. The diagonal adjoint action of
  $K$ on the factors of $\At[1]$ preserves both centres of blow-up, so also
  extends smoothly to $G[2].$ Thus it suffices to consider the product
  where one factor is diagonal
\begin{equation}
\left(t_1^{-\ha}q(\xi)+t_1^{\ha}q^\perp(\xi)\right)\begin{pmatrix}t_2^{\ha}&0\\0&t_2^{-\ha}
\end{pmatrix};
\label{CoSL2.118}\end{equation}
here the second factor has been inverted as in \eqref{CoSL2.173} and, the
center of blow-up being in the corner, we may suppose that both $t_1$ and
$t_2$ are close to zero.

The polar part of the product in \eqref{CoSL2.118} is readily computed
\begin{multline}
a(t_1,t_2)^2=\begin{pmatrix}t_2^{\ha}&0\\0&t_2^{-\ha}
\end{pmatrix}\left(t_1^{-1}q(\xi)+t_1q^\perp(\xi)\right)
\begin{pmatrix}t_2^{\ha}&0\\0&t_2^{-\ha}
\end{pmatrix}\\
=(t_1t_2)^{-1}\bigg(q(e_2)\left(q(\xi)+t_1^2q^\perp(\xi)\right)q(e_2)
+t_2q(e_2)\left(q(\xi)+t_1^2q^\perp(\xi)\right)q(e_1)\\
+t_2q(e_1)\left(q(\xi)+t_1^2q^\perp(\xi)\right)q(e_2)
+t_2^2q(e_1)\left(q(\xi)+t_1^2q^\perp(\xi)\right)q(e_1)\bigg),
\label{CoSL2.119}\end{multline}
written as a sum of the four terms corresponding to the basis $e_1,$ $e_2,$
so each has rank at most one.

If $t_1\downarrow0$ and $t_2\downarrow0$ but $\xi$ is bounded away from
$e_1$ then the first, most singular, term is non-zero and there is
necessarily an eigenvalue which is a positive smooth multiple of
$(t_1t_2)^{-1};$ the other eigenvalue is its inverse. Thus the trace of the
square-root must be of the form $\alpha^{-\ha} (t_1t_2)^{-\ha}$ with
$\alpha >0$ and smooth. It follows that the corresponding rescaled matrix
of trace one satisfies
\begin{equation}
b(t_1,t_2)=\alpha q(e_2)q(\xi)q(e_2)+t_1E_1+t_2E_2.
\label{CoSL2.120}\end{equation}
It is therefore a smooth curve in $\At$ approaching the boundary. A similar
computation shows that the factor in $K$ in the polar decomposition of
\eqref{CoSL2.118} is also smooth down to $t_1=t_2=0.$

This actually proves smoothness of the product without the first blow-up in
\eqref{CoSL2.117}, so it certainly remains smooth after this blow-up but
away from the preimage of $\xi=e_2.$

The blow-up of $t_1=t_2=0$ sets $t_i=\tau_is$ where the $\tau_i$ and $s$
are defining functions for the resulting three boundary faces, moreover
$\tau_1+\tau_2>0$ since the two `old' boundary hypersurfaces no longer
intersect. Then \eqref{CoSL2.119} becomes
\begin{multline}
a(t_1,t_2)^2
=\\
s^{-2}(\tau_1\tau_2)^{-1}\bigg(q(e_2)q(\xi)q(e_2)+s\tau_2\left(q(e_2)q(\xi)q(e_1)+q(e_1)q(\xi)q(e_2)\right)+s^2F\bigg)
\label{CoSL2.122}\end{multline}
with $F$ smooth.

The first term is a multiple $R^2(\xi)q(e_2)$ of the projection onto $e_2$
with coefficient which is the square of a defining function for $\xi=e_1$
and the coefficient of $s$ vanishes at $\xi=e_1.$
The second blow-up, in \eqref{CoSL2.117}, is the introduction of polar
coordinates in the sense that a defining function for the new front face is
$x^2=R^2+s^2.$ Then $s=\sigma x$ where $\sigma$ is a defining function for
the lift of $s=0$ and $R^2=x^2r^2$ where $r^2$ is smooth, non-negative, and
vanishes precisely at the lift of the diagonal $\xi=e_2.$ Thus
\eqref{CoSL2.122} becomes 
\begin{equation}
a(t_1,t_2)^2
=x^{-2}(\tau_1\tau_2)^{-1}\big(r^2q(e_2)+\sigma e+xf)\big)
\label{CoSL2.123}\end{equation}
where all terms are smooth and $e$ is linearly independent of $q(e_2)$ and
does not vanish at $r=0.$

It follows that
\begin{equation}
a(t_1,t_2)=x^{-1}(\tau_1\tau_2)^{-\ha}b(x,\sigma ,\tau_1,\tau_2)
\label{CoSL2.124}\end{equation}
projects to a smooth family in $\At[1].$

Again a similar analysis shows the smoothness of the compact factor, so the
product does extend to a smooth map. A boundary defining function for
the left factor $G[1]$ lifts to the product of boundary defining functions for the three
boundary faces excepting the remaining boundary hypersurface which projects
onto the boubdary of the right factor of $G[1].$ The b-submersion
condition follows from analysis of invariant vector fields; hence the map
is b-fibration.

That the two projections lift to be b-fibrations is standard for the edge
stretched product for any boundary fibration.
\end{proof}

Convolution on $ \CIc(G)$ is given by the standard formula
\begin{equation*}
	(f_1, f_2) \mapsto 
	f_1 *f_2(g)
	= \int_G f_1(gh^{-1})f_2(h)\,dh.
\end{equation*}
This can be interpreted geometrically as
\begin{equation*}
	f_1 * f_2
	= (\pi_L)_*(\chi^*f_1\cdot \pi_R^*f_2\cdot dh)
\end{equation*}
where $\chi(g,h) = gh^{-1}.$ Since it is most natural to push forward
densities we multiply by Haar measure on $G$ and write the convolution
formula as 
\begin{equation}
	f_1 * f_2dg
	= (\pi_L)_*(\chi^*f_1\cdot \pi_R^*f_2\cdot dgdh)
\label{CoSL2.126}\end{equation}

A basic result due to Harish-Chandra which follows from this geometric
setup is:-

\begin{proposition}\label{CoSL2.174} 
Convolution extends by density from $\CIc(G) \subseteq \HC(G)$ to
\begin{equation}
	\HC(G) \times \HC(G) \lra \HC(G) 
\label{CoSL2.136}\end{equation}
\end{proposition}

The proof, below, depends on an examination of the functions and measures
in \eqref{CoSL2.126}. The notion of a b-fibration and some of the
properties of such maps are briefly recalled in the appendix. Note that for
a b-fibration all hypersurfaces in the domain are of one of two types,
either `fixed' -- those which are mapped onto the image space -- or
`non-fixed' if mapped into (and then necessarily onto) a boundary
hypersurface.

\begin{lemma}\label{CoSL2.175} If $f:X\longrightarrow Y$ is a b-fibration
  between compact manifolds with corners then pull-back defines a
  continuous map 
\begin{equation}
f^*:(\ilog\rho )^{\infty}\cA(Y)\longrightarrow (\ilog \rho')^{\infty}\iota_{H''}\cA(X)
\label{CoSL2.176}\end{equation}
where $\rho$ is a total boundary defining function on $Y,$ $\rho '$ is a
collective boundary defining function for the non-fixed hypersurfaces in
$X$ and $\iota_{H''}$ is the formal weight denoting smoothness up to the
fixed hypersurfaces. Similarly 
\begin{equation}
f_*:(\ilog\rho_{H'})^{\infty}w_{H''}\cA(X;\Omega _{\bo}))\longrightarrow
(\ilog \rho')^{\infty}\cA(Y;\Omega _{\bo}))
\label{CoSL2.177}\end{equation}
provided $w_{H''}$ is an integrable weight at the fixed boundary
hypersurfaces.
\end{lemma}

\begin{proof} Under pull-back with respect to a b-map conormal
  functions with weight $w$ lift to be conormal with weight $f^*w$ at the
  non-fixed hypersurfaces and with smoothness up to the fixed
  hypersurfaces. This gives \eqref{CoSL2.176} since the pull-back of the
  logarithmic weight $\ilog\rho$ at a hypersurface $H'$ in the base
  satisfies
\begin{equation}
c\prod_{f(H)=H'}(\ilog\rho _H)\le f^*\ilog\rho _{H'}\le
C\prod_{f(H)=H'}(\ilog\rho _H)^{1/p(H')},\ c,\ C>0,
\label{CoSL2.178}\end{equation}
where $p(H')$ is the multiplicity of $f$ at $H',$ the maximum number of hypersurfaces in the
preimage of $H'$ with non-empty mutual intersection.

Under push-forward, there is in general a fixed `logarithmic growth' of order $p(H')$ in
essentially the same sense. That is 
\begin{multline}
f_*:\prod_{H'\in\cM_1(Y)}f^*(\ilog\rho_{H'})^kw_{H''}\cA(X;\Omega _{\bo}))\longrightarrow\\
\prod_{H'\in\cM_1(Y)}(\ilog\rho_{H'})^{k-p(H')}\cA(Y;\Omega _{\bo})).
\label{CoSL2.179}\end{multline}
Using \eqref{CoSL2.178} again in both domain and range, \eqref{CoSL2.177} follows. 
\end{proof}

Note that if $X$ is any compact manifold with corners the space of
log-rapid decaying conormal functions can be defined in two ways, since
\begin{equation}
(\prod_{H\in\cM_1(X)}\ilog\rho _H)w^\infty\cA(X)=
(\ilog\rho)^\infty\cA(X),\ \rho=\prod_{H\in\cM_1(X)}\rho _H.
\label{CoSL2.190}\end{equation}

\begin{proof}[Proof of Proposition~\ref{CoSL2.174}]

The first step is to analyze the boundary behaviour of the product $dgdh$
of the Haar measures on $G[2].$ On $G[1]^2$ it is an edge density
\begin{equation}
dgdh=t_1^{-2\kappa }t_2^{-2\kappa}\nu_{\bo}
\label{CoSL2.127}\end{equation}
where $\nu_{\bo}$ is a positive b-density. Under blow up of a boundary
face, in this case the corner, a positive b-density lifts to a positive
b-density so 
\begin{equation}
dgdh=\tau_1^{-2\kappa }\tau_2^{-2\kappa}s^{-4\kappa}\nu_{\bo}\Mon [K^2\times\At[2,\bo]).
\label{CoSL2.128}\end{equation}
with $\tau_1,$ $\tau_2$ and $s$ defining functions for the three boundary
hypersurfaces. The second blow up is of a p-submanifold of $s=0$ of
codimension $2\kappa,$ i.e.\ the dimension of $\pa\At[1].$ It follows that 
\begin{equation}
dgdh=\tau_1^{-2\kappa }\tau_2^{-2\kappa}s^{-4\kappa}x^{-2\kappa}\nu_{\bo}\Mon G[2].
\label{CoSL2.129}\end{equation}
This is essentially the formula for Lebesgue measure in polar coordinates.

Next consider the pull-back $\pi_R^*f_2$ of an element of $\HC(G)$ to
$G[2]$ under the b-fibration $\pi_R.$ Lemma~\ref{CoSL2.175} shows that for
the pull-back of bounded conormal functions on a compact manifold with
boundary, log-rapid decrease is reflected in log-rapid decay at all
non-fixed hypersurfaces and smoothness up to the `fixed'
hypersurfaces. Since
\begin{equation}
\pi_R^*t_2=a\tau_2 sx
\label{CoSL2.130}\end{equation}
where $a>0$ the linear weight $t_2$ lifts to the product of the three
weights so
\begin{equation}
\pi_R^*f_2\in (\ilog \tau_2)^\infty(\ilog s)^\infty(\ilog
x)^\infty(\tau_2sx)^\kappa\iota_{\tau_1=0}\cA(G[2]),\ f_2\in\HC(G[1]).
\label{CoSL2.133}\end{equation}

Now essentially the same analysis shows that  
\begin{equation}
\chi^*f_1\in(\ilog \tau_1)^\infty(\ilog \tau_2)^\infty(\ilog
s)^\infty(\tau_1s^{2}\tau_2)^\kappa\cA(G[2]),\ f_1\in\HC(G[1]).
\label{CoSL2.134}\end{equation}
It follows that the product  
\begin{equation}
\begin{gathered}
(\chi^*f_1)(\pi_R^*f_2)\in(\ilog \tau_1\ilog
\tau_2\ilog s\log x)^\infty(\tau_1s^{3}\tau_2^{2}x)^\kappa 
\cA(G[2])\Longrightarrow \\
(\chi^*f_1)(\pi_R^*f_2)dgdh\in(\ilog \tau_1\ilog
\tau_2\ilog s\log x)^\infty(\tau_1sx)^{-\kappa} 
\cA(G[2];\Omega_{\bo}).
\end{gathered}
\label{CoSL2.135}\end{equation}

The absence of any power weight in $\tau_2,$ with the log-rapid decay,
means that the density \eqref{CoSL2.135} is fibre-integrable for $\pi_L,$
so again using Lemma~\ref{CoSL2.175} the push-forward of the product is
well-defined and in the space $\HC(G[1])dg,$ which is Harish-Chandras's
result, \eqref{CoSL2.136}.
\end{proof}

\section{Iwasawa decomposition}\label{S.Iwa}

As noted above, the base of the fibration given by the product
decomposition 
\begin{equation}
\pa G[1]=\mK\times \pa \At[1]
\label{CoSL2.268}\end{equation}
can be identified with the space of parabolic subgroups of $\SL(2).$
Each parabolic is conjugate to the upper triangular subgroup
$\Pa_+\subset\SL(2)$ so $\pa\At[1]$ is identified as the 1-dimensional
real, respectively complex, projective space. The boundary of $\At[1],$ realized as the
rank one positive matrices, is given by the corresponding projections. In
this identification, $\Pa_+$ is identified with the line $[e_1],$ or the
corresponding projection, $q(e_1).$

Consider the Iwasawa decomposition of $\SL(2)$
\begin{equation*}
	G = \mK\pDg\uM_+.
\end{equation*}
The quotient $G/\uM_+$ can therefore be identified with $\mK\pDg$ where
$\pDg$ is the subgroup of positive definite diagonal matrices. This leads to
the induced compactification given by the closure $\pDg[1]$ of $\pDg$ in
$\At[1]$
\begin{equation}
	(G/\uM_+)[1] = \mK\times \pDg[1] \subseteq G[1]. 
\label{CoSL2.137}\end{equation}
Thus, for $\SL(2),$ $\pDg[1]\subset\At[1]$ is the same closed interval for both complex
and real cases.

The relationship between the base of the boundary fibration
\eqref{CoSL2.268} and the parabolic subgroups can be seen more geometrically.

\begin{lemma}\label{CoSL2.269} The closure in $G[1]$ of each orbit for the
  right action of $\uP$ on $G$ contains two points of $\pa
  G[1]=\mK\times\pa\At[1]$ of the form $\pm kq(N)$ for a fixed point
  $q_1(\uP)\in\pa\At[1]$ and some $k\in\mK.$
\end{lemma}

\begin{proof} Since the parabolic subgroups are all conjugate, it suffices
  to consider the right action of $\uP_+.$ The action of $K$
  on the left commutes with composition with $n(x)\in\uP_+$ on the right, so it
  suffices to consider the orbit starting at a point of $\Af.$ In fact from
  the Iwasawa decomposition it is enough to consider initial points in
  $\pDg.$ Consider the polar decomposition of the corresponding curve
\begin{equation*}
g(x;\tau)=\begin{pmatrix}
\tau^{\ha}&0\\0&\tau^{-\ha}
\end{pmatrix}
\begin{pmatrix}
1&-x\\0&1
\end{pmatrix}=
\begin{pmatrix}c&s\\-s&c\end{pmatrix}\begin{pmatrix}a&b\\b&d\end{pmatrix}=
k(x;\tau)a(x;\tau).
\end{equation*}
Since $a$ is positive definite and
$$
a(x;\tau)^2=\begin{pmatrix}\tau&-x\tau\\
-x\tau&x^2\tau+\tau^{-1}\end{pmatrix},
$$
it follows that
\begin{equation}
a(x;\tau)=\frac{1}{(\tau+x^2\tau+\tau^{-1}+2)^{\ha}}\begin{pmatrix}
\tau+1&-x\tau\\-x\tau&x^2\tau+\tau^{-1}+1\end{pmatrix}.
\label{CoSL2.180}\end{equation}
Dividing by the trace it follows that the image curve is smooth in $\At[1]$
in terms of the radially compactified variable $x/(1+x^2)^{\ha}$ and 
\begin{equation}
a(x;\tau)\to\begin{pmatrix}0&0\\0&1
\end{pmatrix}
\label{CoSL2.270}\end{equation}
at both end-points. Similarly
\begin{equation}
c=\frac{\tau+1}{((\tau+1)^2+
(x\tau)^2)^{\ha}},\
s=\frac{-x\tau}{((\tau+1)^2+(x\tau)^2)^{\ha}}, 
\label{CoSL2.221}\end{equation}
which are smooth functions of $1/|x|$ as $x\to\infty$ and
\begin{equation}
k(\tau;x)\to\pm\begin{pmatrix}0&1\\-1&0\end{pmatrix}\Mas x\to\mp\infty.
\end{equation}
\end{proof}

Notice that the limit point of the right $\uP$ orbits projected to $\At$
\begin{equation}
q_1(\uP_+)=e_2
\label{CoSL2.271}\end{equation}
is the opposite end-point of $\pDg[1]$ to that coming from the
limiting projection above, which will be denoted $q_0(\uP_+)=e_1.$

We define a second compactification of $G$ associated to a choice of
parabolic by two levels of blow-up, associated to $q_1(\uP),$ from $G[1]$ 

\begin{definition} The compactification of $G$ relative to $\uP,$ acting on the right, is
\begin{equation}
	G[1;\uP] = K\times \At[1;\uP],\ \At[1;N]=[[\At[1];\{q(\uP)\}]; \pa\ff]. 
\label{CoSL2.201}\end{equation}
Here $\pa\ff$ is the codimension two corner which is the boundary of the
front face introduced in the first blow-up, so consisting of two points
for $\SL(2,\bbR)$ and a circle for $\SL(2,\bbC).$
\end{definition}

Thus, $G[1;\uP]$ has three bounding hypersurfaces, an `old' one, numbered
`$0$' corresponding to the original boundary, and the two front faces
numbered `$1$' and `$2$' from the two blow-ups. So the `2' face separates
the `0' and `1' faces.

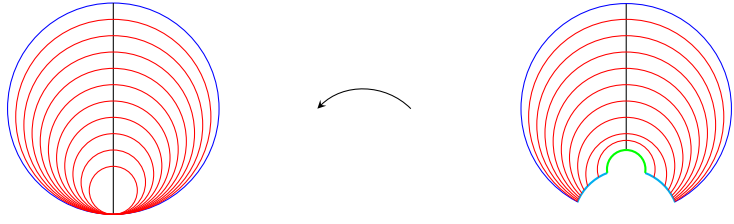
\begin{figure} 
\begin{tikzcd}
\begin{tikzpicture}
\def\R{40};
\def\n{10}
\def\h{7/10*\R};
\draw [blue] (\R pt,\R pt) circle [radius=\R pt];
\draw (40pt, 0pt) -- (40pt, 80pt);
\foreach \x in {1,2,...,\n}
	\pgfmathsetmacro\ycoord{40-(\R+\h)/2*(1-\x/(\n +1))}
	\pgfmathsetmacro\radius{\ycoord}
	\draw [red] (40 pt,\ycoord pt) circle [radius=\radius pt];
\end{tikzpicture}
&
\begin{tikzpicture}[>=stealth]
\node at (0,0) {};
\draw [<-] (0,40pt) arc [start angle=135, end angle=45, radius=25pt];
\node at (40pt,0) {};
\end{tikzpicture}
&
\begin{tikzpicture}[remember picture,overlay]

\tikzstyle{reverseclip}=[insert path={(current page.north east) --
  (current page.south east) --
  (current page.south west) --
  (current page.north west) --
  (current page.north east)}
]

\def\R{40};
\def\n{10} 
\def\h{7/10*\R};
\path [name path=orcircle] (\R pt,\R pt) circle [radius=\R pt];
\pgfmathsetmacro\smallr{2 *\R / (\n+1)};
\pgfmathsetmacro\smally{\R-\h*(1-2/(\n +1))};
\coordinate (smallc) at (\R pt, \smally pt);
\path [name path=1blow] (smallc) circle [radius=\smallr pt];
\def\bigr{20}
\pgfmathsetmacro\bigy{\R -(\h+\bigr-5)};
\coordinate (bigc) at (\R pt, \bigy pt);
\path [name path=2blow]  (bigc) circle [radius=\bigr pt];
\path [name path=diag] (40 pt, 0 pt) -- (40 pt, 80 pt);

\path[name intersections = {of = diag and 1blow, by = {botd, botd'}}];
\draw (botd) -- (40pt, 80pt);

\clip (40pt,40pt) circle [radius=40pt];
\begin{scope}
\begin{pgfinterruptboundingbox} 
\path [clip] (smallc) circle [radius=\smallr pt] [reverseclip];
\end{pgfinterruptboundingbox}
\draw [cyan, thick] (bigc) circle [radius=\bigr pt];
\end{scope}

\begin{scope}
\begin{pgfinterruptboundingbox} 
\path [clip] (bigc) circle [radius=\bigr pt] [reverseclip];
\end{pgfinterruptboundingbox}

\draw [blue, thick] (40pt,40pt) circle [radius=40pt];
\foreach \x in {3,4,...,\n}
	\pgfmathsetmacro\ycoord{40-(\R+\h)/2*(1-\x/(\n +1))}
	\pgfmathsetmacro\radius{\ycoord}
	\draw [red] (40 pt,\ycoord pt) circle [radius=\radius pt];

\draw [green, thick] (smallc) circle [radius=\smallr pt];
\draw [red, scale=1.5] (smallc) circle [radius=\smallr pt];

\end{scope}

\end{tikzpicture}
& & & 
\end{tikzcd}
\caption{Resolving the unipotent flow}\label{fig1}
\end{figure}

\begin{proposition}\label{G1N} The quotient map $G\longrightarrow G/N$
  extends to a fibration
\begin{equation}
	\xymatrix{
	&G \ar[r]^{\pi_\uP} \ar@{^(->}[d] & G/\uP \ar@{^(->}[d] \\
	\uP[1]\ar@{-}[r] & G[1;\uP] \ar[r]^{\pi_N}\ar[d]^{\beta} & G/\uP[1]\\
&G[1]
}
\label{CoSL2.294}\end{equation}
where $\uP[1]$ is the radial compactification of $\uP$ as a Euclidean
space. Defining functions $\rho _i,$ $i=0,1$ for the boundaries of $G/N[1]$
and $\rho$ for the boundary of $G[1]$ pull back in terms of defining functions for the boundaries of $G[1;\uP]$ as
\begin{equation}
\pi_\uP^*\rho_0=\tilde\rho_0,\ \pi_\uP^*\rho_1=\tilde\rho_1\tilde\rho_2^2,\ \beta
^*\rho=\tilde\rho_0\tilde\rho_1\tilde\rho_2^2 
\label{CoSL2.272}\end{equation}
and the generating vector field for the right action of $\uP$
lifts to $G[1;N]$ to be of the form 
\begin{equation}
V(\uP)=\tilde\rho _1\tilde\rho_2^2W,\ 0\not=W\text{ smooth },\
W\tilde\rho_2\not=0\Mat\tilde\rho _2=0
\label{CoSL2.273}\end{equation}
where $W$ is tangent to the boundary surfaces $\{\tilde\rho_0=0\},$
$\{\tilde\rho_1=0\}.$ The inclusion \eqref{CoSL2.137} makes $G/\uP[1]$ into a
p-submanifold transversal to the fibration such that $\tilde\rho_0$ and $\tilde\rho
_1$ restrict to boundary defining functions.
\end{proposition}
\noindent So this blow-up `resolves' the right action of $\uP$ in the sense
that the closures of the the orbits in $G$ become the orbits of the smooth
vector field $W$ and are precisely the fibres of $\pi_\uP.$ Note however
that while the action of $\uP$ does extend smoothly to $G[1;\uP]$ the
points on the two front faces are all fixed points for the action. So
although this gives a meaning to the `quotient' formula
\begin{equation*}
G[1;N]/N[1]=G/N[1]
\label{CoSL2.139}\end{equation*}
this is not strictly correct in terms of orbit spaces.

\begin{proof} The conjugation action by $\mK$ allows the discussion to be
  reduced to the case of $\uP_+\subset G=\SL(2,\bbR)$ We proceed to compute
  the form of the generating vector field for the right action of $\uP_+,$
  which we know to be smooth on $G[1].$ In fact we have already seen that
  $V(N)$ is tangent to the boundary and non-vanishing except at
  $\mK\times\{q_1(N)\}.$

A neighborhood of $q_1(N)\in\At[1]$ is smoothly parameterized by the matrices 
\begin{equation}
\beta (t,s)=\begin{pmatrix}t+s^2&s\\s&1\end{pmatrix},\ t\ge0,\ |s|<\delta .
\label{CoSL2.288}\end{equation}
Applying $\uP_+$ on the right and taking the polar decomposition gives 
\begin{equation*}
\beta\begin{pmatrix}1&-x\\0&1\end{pmatrix}=KA,
\label{CoSL2.289}\end{equation*}
where
\begin{equation*}
\begin{gathered}
K=T\begin{pmatrix}t+s^2-xs+1&-x(t+s^2)\\x(t+s^2)&t+s^2-xs+1\end{pmatrix},
\\
\begin{aligned}
&A=\\
&T\begin{pmatrix}(t+s^2)^2+s^2+t&-x(t+s^2)^2+s(t+s^2-sx+1)\\
-x(t+s^2)^2+s(t+s^2-sx+1)&(-x(t+s^2)+s)^2+(-sx+1)^2+t\end{pmatrix},
\end{aligned}
\\
T=T(x,s,t)=((x(t+s^2))^2+(t+s^2-xs+1)^2)^{\ha}.
\end{gathered}
\end{equation*}

Differentiating at $x=0$ gives the generating vector field 
\begin{equation}
V(N)=(s^2+t+1)^{-1}\left(2st(2s^2+2t+1)\pa_t+(s^4+s^2-t^2)\pa_s-(s^2+t)\pa_{\theta}\right).
\label{CoSL2.287}\end{equation}
In particular all coefficients vanish at $q_1(N).$

The blow up of the boundary point introduces a boundary face which is an
interval. A neighbourhood of this boundary hypersurface of $[B[1];q_1(\uP_+)]$
is covered by the three coordinate systems 
\begin{equation}
\frac{s}t,t\Mover |s|<2t,\ \frac{t}{|s|},|s|\Mover t<2|s|
\label{CoSL2.186}\end{equation}
where the second pair of coordinate systems cover the endpoints and the
first actually covers the interior. From \eqref{CoSL2.180} and
\eqref{CoSL2.221} it follows that along the integral curves
\begin{equation}
c_{\tau}s^2<|t|<G_{\tau}s^2,\ \tau>0.
\label{CoSL2.187}\end{equation}
So after this first blow-up these curves approach one of the end-points of
the front face.

The second blow-up replaces each of these end-points by an interval with a
neighbourhodd covered by two coordinate systems. Restricting attention to
the end-point in $t<4s$ these are given by 
\begin{equation}
\eta=\frac{t}{s^2},\ s\Mover t<2s^2,\ \sigma=\frac{s^2}t,\ \frac{t}{s}
\Mover 2t>s^2.
\label{CoSL2.188}\end{equation}
In particular near points of the interior of the interval, either system is
admissible. It follows from \eqref{CoSL2.187} that each integral curve of
$N$ starting at a finite point, $\tau>0,$ hits the boundary in the interior
of this second front face at a unique point. Indeed along the curves defined by $b(x,\tau)$ 
\begin{equation}
\eta\longrightarrow \tau^{-1},\ s=|x|^{-1}+O(|x|^{-3})
\label{CoSL2.189}\end{equation}
and $\eta$ and $s$ are smooth functions of $x^{-1}$ and $0<\tau<\infty.$

From \eqref{CoSL2.287}, in polar coordinates at $t=0,$ $s=0,$ $S=s/t,$ 
\begin{multline*}
V=(s^2+t+1)^{-1}t\big((2(2s^2+2t+1))(S(t\pa_t-S\pa_S)\\
+(S^4t^2+S^2-1)\pa_S-(1+S^2)\pa_\theta \big)=tW
\label{CoSL2.223}\end{multline*}
where $W\big|_{t=0}\not=0.$ Making the second blow-up introduces $\eta=t/s^2$ and
$s'=s$ in $t<<s^2$ and $\sigma =t^2/s,$ $T=t/s$ in $t>>s^2$ in terms of which 
\begin{equation}
2t\pa_t+s\pa_s=s'\pa_{s'},\ t\pa_t=\eta\pa_\eta\Longrightarrow  V=(s')^2W,\ Ws'\not=0.
\label{18.10.2018.5}\end{equation}
Similar analysis in the other regions gives the stated form of $W.$
\end{proof}

It is convenient to have a `universal' version of this resolution. If we
consider the space $G[1]\times\mK$ as a bundle over $\mK$ then in each
fibre we may consider the action of the parabolic group parameterized by
$k'\in\mK$ corresponding to the unipotent group
\begin{equation*}
\uP_{k'}=k'\uP_+(k')^{-1}.
\label{SL2.4}\end{equation*}
The total space of the fibre bundle in which the action of this group in each fibre is
resolved, as above, is a compact manifold with corners
\begin{equation}
(G\times\mK)[1;\uP_*]=[G[1]\times \mK;q_1(\uP_*);\pa_{\ff}]
\label{SL2.5}\end{equation}
where the p-submanifold $q_1(\uP_*)$ is the graph of $\mK\ni k'\longmapsto
q_1(k')\in\pa G[1]$ and the second blow-up is of its boundary. The
`universal' quotient map corresponding to the collective fibration by the
$\uP_*$ is
\begin{equation}
\xymatrix{
(G\times\mK)[1;\uP_*]\ar[r]^-{/\uP_*}&(G\times\mK)/\uP_*[1]\ar[r]^-{\pi_{\mK}}&\mK
}
\label{SL2.6}\end{equation}
where the central space is the fibre bundle over $\mK$ with fibre at $k'$
the compactified space $(\mK\times k'\pDg (k')^{-1})[1]\subset G[1]$ of
$k'$-conjugates of diagonal matrices. This fibrewise action gives a
diffeomorphism to the product bundle over the last factor
\begin{equation}
\begin{gathered}
(G\times\mK)/\uP_*[1]\longrightarrow \mK\times\pDg[1]\times\mK\text{ extending}\\
G\times\mK/\uP_*\ni(k,q',k')\longmapsto (kk',(k')^{-1}q'k',k')\in\mK\times\pDg\times\mK.
\end{gathered}
\label{SL2.7}\end{equation}

The transversality of this action of $\mK$ on $G\times\mK$ means that the
projection back is a b-fibration 
\begin{equation}
\pi_L:(G\times\mK)[1;N_*]\longrightarrow G[1].
\label{SL2.9}\end{equation}

Following standard prescriptions  

\begin{lemma}\label{HCGN}
The Harish-Chandra space of $G/\uP$ is identified with the conormal space
\begin{equation}
\HC(G/\uP)=(\ilog\rho )^{\infty}\rho_0^\kappa\rho_1^{-\kappa}\cA(G/\uP[1])
\label{CoSL2.191}\end{equation}
where $\rho=\rho_0\rho_1$ is a total boundary defining function and
$\rho_1$ defines the end corresponding to $q(\uP).$
boundaries respectively. 
\end{lemma}

\begin{proof} The invariant vector field on $\pDg$ is the radial vector
  field $\tau\pa_\tau$ so the span with the generating vector field(s) from
  $\mK$ gives b-regularity with respect to the weight $\delta^{-\ha}$ plus log-rapid decay.
\end{proof}

Having established that $G[1;\uP]$ fibres over $G/\uP[1],$ consider the
pull-back via the total blow-down map, $\beta :G[1;\uP]\longrightarrow G[1],$
of $\HC(G[1])$ to $G[1;\uP].$ In fact 
\begin{equation}
\beta^*\HC(G)\subset (\log\rho)^\infty(\rho_o\rho_1\rho_2^2)^\kappa\cA(G[1;\uP]).
\label{CoSL2.211}\end{equation}
where the weight follows from \eqref{CoSL2.272}. This is not an equality of
spaces since the pulled back functions have more regularity at the front faces.

The properties of the compactification $(G\times\mK)[1;N_*]$ also allow us
to analyse the boundary behaviour of Harish-Chandra's spherical function
for $\SL(2,\bbK).$ The spherical function is defined in terms of the
pseudocharacter given by the eigenvalue quotient
\begin{equation}
\delta :\pDg\ni\begin{pmatrix}
\tau^{-\ha}&0\\
0&\tau^{\ha}
\end{pmatrix}\longrightarrow \tau^{-1}\in[0,\infty)
\label{CoSL2.231}\end{equation}
extended to $G$ through the Iwasawa decomposition
\begin{equation}
\delta :G\longrightarrow\pDg\longrightarrow (0,\infty).
\label{CoSL2.232}\end{equation}
Then the spherical function (associated to $N)$ is 
\begin{equation}
\Phi(g)=\frac1{\Vol(\mK)}\int_\mK\delta ^{-\ha}(gk^{-1})dk.
\label{CoSL2.233}\end{equation}
Notice that the invariant measure on $G/\uP_+=\mK\pDg$ is
$\delta^{-1}dkda.$

\begin{lemma}\label{CoSL2.234} The spherical function for $\SL(2)$ with
  respect to $\uP$ is polyhomogenous conormal, positive, bi-invariant for
  the action of $\mK$ and takes the form near the boundary
\begin{equation}
\Phi(g)=-bt^{\mu}\log t+at^{\mu},\ a,b\text{ smooth},\ b\big|_{\pa G[1]}>0.
\label{CoSL2.235}\end{equation}
\end{lemma}\noindent
In particular, as a weight function, $1/\Phi,$ with rapid logarithmic decay
added, gives the same space as $\rho^{\mu}$ on $G[1].$

\begin{proof} Consider the real case for $\uP_+.$ Since $\delta$
is pulled-pack from $\pDg$ using the Iwasawa decomposition with respect to
$\uP$ the integrand in \eqref{CoSL2.233} as a function on $G\times\mK$ is,
for each $k\in\mK,$ the pull-back of the corresponding function on the
$k'$-diagonal matrices under the Iwasawa decomposition for
$\uP_{k}=k\uP k^{-1}.$ These actions are resolved on $(G\times\mK)[1;N_*]$
and it follows that
\begin{equation}
\delta^{-1}(gk^{-1})=\tilde\rho _0^{-1}\tilde\rho _1\Mon (G\times\mK)[1;\uP_*]
\label{CoSL2.236}\end{equation}
is the product of the inverse of a defining functions for the `old'
boundary and a defining function for the first front face.

Thus $\Phi$ is the push-forward to $G[1]$ of
$\tilde\rho_0^{\ha}\tilde\rho_1^{-\ha}$ with respect to the fibre density
$dk$ under the b-fibration \eqref{SL2.9}. To compute the form of $\Phi,$
choose a positive b-density $\nu_G$ on $G[1].$ Then $\nu_Gdk$ is a positive
b-density on $G[1]\times K$ and lifts after the first blow-up to be of the
form $\rho _1\nu_{\bo}.$ Since the second blow-up is of a corner, which is
a boundary hypersurface of the first front face, the lift to $(G\times
K)[1;\uP_*]$ is of the form $\rho _1\rho _2\nu_{\bo}.$ Thus the spherical
function satisfies
\begin{equation}
\Phi\nu_G=(\pi_G)_* \rho _0^{\ha}\rho_2\rho_1^{\ha}\nu_{\bo}\Longrightarrow
\Phi=at^{\ha}-bt^{\ha}\log t, a,\ b\text{ smooth,} b>0\text{ near }\pa G[1].
\label{CoSL2.237}\end{equation}
Here the positivity follows from the positivity of the integrand and the
coefficient of the logarithm corresponds precisely to the integral over the
corner.

The argument for the spherical function on $\SL(2,\bbC)$ with respect to
the upper triangular Borel subgroup is very similar.
\end{proof}

\begin{lemma}\label{CoSL2.292} Averaging over a unipotent subgroup gives a
  continuous linear map 
\begin{equation}
\int_Ndn:\HC(G)\longrightarrow \HC(G/N).
\label{CoSL2.293}\end{equation}
\end{lemma}

\begin{proof} This follows from the properties of pull-back and
  push-forward for the space $(G\times\mK)[1;N_*].$. Following
  the lifting property \eqref{CoSL2.211} it suffices to show that along the
  central row in \eqref{CoSL2.294} push-forward gives
\begin{equation}
\int_Ndn:(\log\rho)^\infty(\rho_0\rho_1\rho_2^2)^\kappa\cA(G[1;\uP])\longrightarrow
(\log\rho)^\infty\rho_0^\kappa\rho_1^{-\kappa}\cA(G/N[1]).
\label{CoSL2.295}\end{equation}
It follows from Proposition~\ref{G1N} that
$dn=\rho_1^{-2\kappa}\rho_2^{-2\kappa} d\bar n_{\bo},$ where $d\bar n_{\bo}$
is a non-vanishing b-measure on the closed interval $N[1]$ -- which is
transversal to the boundary $\{\rho_2=0\}.$ Thus the decay at $\rho_2=0$ in
\eqref{CoSL2.211} is indeed sufficient to give integrability across the
fixed hypersurface $\{\rho_2=0\}$ and the extra factor of
$\rho_1^{-2\kappa}$ gives the change of weight in \eqref{CoSL2.295}.
\end{proof}

It is immediate that the map \eqref{CoSL2.295} is surjective but in fact
\eqref{CoSL2.293} is also surjective although this is not so elementary.

\section{$\HC(G/N)$ as a module}\label{S.mod}

For a unipotent subgroup and
\begin{multline}
f\in\CIc(G),\ u\in\CIc(G/\uP),\\
f*u(h)=\int_Gf(g)u(g^{-1}h)dg=\int_Gf(hg^{-1})u(g)dg\in\CIc(G/\uP)
\label{CoSL2.198}\end{multline}
defines a continuous linear map
\begin{equation}
f*:\CIc(G/\uP)\longrightarrow \CIc(G/\uP).
\label{CoSL2.199}\end{equation}
For the upper triangular case, fixing $h=kq,$ $(k,q)\in\mK\pDg,$
$$
f*u(k,q)=\int_{\mK\times\pDg}S(f)(k,q,k',q')u(k',q')dk'dq'
$$
where for fixed $k$ and $q$
\begin{multline}
S(f)(k,q;k',q')\\
=(\pi_N)_*f(k(q(q')^{-1})(n')^{-1})(k')^{-1})u(k',q')\delta(q')^{-1}dn=\tilde
S(f)(k\tilde q(k')^{-1}),\\
n'=q'n(q')^{-1},\ \tilde q=q(q')^{-1}
\label{CoSL2.277}\end{multline}
and
\begin{equation*}
dn'=\delta(q')dn\Longrightarrow 
\tilde S(f)(k,\tilde q,k')
=(\pi_N)_*f(k\tilde qn^{-1}(k')^{-1})\Mon \mK\pDg\times\mK\pDg
\label{CoSL2.291}\end{equation*}
where $\pi_N$ is the push-forward map in Lemma~\ref{CoSL2.292}.

This corresponds to the resolution given by the space
$(G\times\mK)[1;\uP_*].$ Thus it follows as in Lemma~\ref{CoSL2.292} that
\begin{equation}
\tilde S:\HC(G)\longrightarrow 
(\ilog\rho)^{\infty}\rho_1^{-\kappa}\rho_0^{\kappa}\cA(\mK\times\pDg[1]\times\mK)
\label{CoSL2.220}\end{equation}
where we have used the push-forward theorem for the b-fibration \eqref{SL2.7}.
\begin{proposition}\label{CoSL2.218} The map \eqref{CoSL2.199} extends by
  continuity to a bilinear map
\begin{equation}
\HC(G)\times\HC(G/N)\longrightarrow \HC(G/N)
\label{14.11.2018.17}\end{equation}
making $\HC(G/N)$ a module over $\HC(G)$ extending the product
\eqref{CoSL2.199}; the action of $\HC(G)$ is through a family of $\bbR_+$-invariant
b-smoothing operators on $G/N[1].$
\end{proposition}

\begin{proof} From \eqref{CoSL2.220} it follows that the kernel map
  \eqref{CoSL2.277} has image in the corresponding space of conormal densities
\begin{equation}
S:\HC(G)\longrightarrow
(\log\rho')^{\infty}\rho^{\kappa}_{0R}\rho^{\kappa}_{0L}\rho^{-\kappa}_{1R}
\rho_{1L}^{-\kappa}\iota''\cA(G/N[2;\bo])dk'dq'
\label{CoSL2.219}\end{equation}
where $\rho '$ is a collective defining function for the `old boundaries'
and similarly $\iota''$ corresponds to smoothness up to the two front
faces. Now the mapping property \eqref{14.11.2018.17} is a direct
consequence of the action diagram for the b-calculus
\begin{equation}
\xymatrix{
&G/N[2]\ar_{\pi_L}[dl]\ar^{\pi_R}[dr]\\
G/N[1]&&G/N[1].
}
\end{equation}
Namely the product of the kernel and the pull-back of an element of
$\HC(G/N)$ from the right is a density in
\begin{equation*}
(\log\rho)^{\infty}\rho^{\kappa}_{1R}
\rho_{1L}^{\kappa}\rho_{\ff0}^{\kappa}\rho^{-\kappa}_{\ff1}\iota''\cA(G/N[2;\bo])dk'dq'
\end{equation*}
where now $\rho$ is a total boundary defining function. This pushes forward
into $\HC(G/N)$ giving \eqref{14.11.2018.17}.
\end{proof}

The natural action of the diagonal group $\pDg$ on $\CIc(G/N)$
includes the pseudocharacter
\begin{equation}
\CIc(G/N)\times \pDg\ni (u,\lambda)\longmapsto \delta^{-\ha}(\lambda)u(k\mu\lambda^{-1})
\end{equation}
where $G/N$ is identified with $\mK\pDg.$ The convolution action of
$v\in\CIc(\pDg)$ is therefore
\begin{equation}
u*v(k\mu)=\int u(k\mu\lambda^{-1})v(\lambda)\delta(\lambda)^{-\ha}d\lambda.
\label{SL2.2}\end{equation}

\begin{lemma}\label{SL2.1} The product \eqref{SL2.2} extends by continuity
  to a jointly continuous bilinear map
\begin{equation}
\HC(G/N)\times\HC(\pDg)\longrightarrow \HC(G/N).
\end{equation}
\end{lemma}

\begin{proof} For the radial compactification of the group $\bbR_+$ the Harish-Chandra
space is 
\begin{equation}
\HC(\pDg)=(\ilog\rho)^{\infty}\cA(\pDg[1]).
\label{SL2.1a}\end{equation}
The composition can be realized in terms of the diagram of b-fibrations
\begin{equation}
\xymatrix{
G/N[1]\\
G/N[2]\ar[r]^{\pi_R}\ar[u]^{\pi_L}\ar[d]_{\chi}&G/N[1]\\
\pDg[2]
}
\label{SL2.2a}\end{equation}
where $\pDg[2]=\pDg[2,\bo]$ and the lower map is the lift of the product
$(q,q')\longmapsto q(q')^{-1}.$

\end{proof}

\section{Parabolic induction}\label{S.ind}

By definition, a tempered representation of a reductive group $G$ is a
smooth representation in a Fr\'echet space $V$ -- so a smooth map 
\begin{equation}
\pi:G\longrightarrow \Hom(V),\ \pi(gh)=\pi(g)\pi(h)
\label{CoSL2.253}\end{equation}
which has the regularity property that the convolution integral 
\begin{equation}
\CIc(G)\ni f\longrightarrow \int _G\phi(gh^{-1})\pi(h)v
\label{CoSL2.252}\end{equation}
extends by continuity to a jointly continuous bilinear map 
\begin{equation}
\tilde\pi:\HC(G)\times V\longrightarrow V
\label{CoSL2.254}\end{equation}
which is a module over convolution 
\begin{equation}
\tilde\pi(f*g,v)=\tilde\pi(f,\tilde\pi(g))
\label{CoSL2.255}\end{equation}
and is surjective 
\begin{equation}
\tilde\pi(\HC(G),V)=V.
\label{CoSL2.256}\end{equation}
In fact this last property is a consequence of the others. Conversely,
\eqref{CoSL2.254}, \eqref{CoSL2.255} an \eqref{CoSL2.256} (apparently)
imply \eqref{CoSL2.253} exists so that \eqref{CoSL2.254} is recovered from
\eqref{CoSL2.252}.

Now, we wish to consider the functor of parabolic induction -- construction
of representations of $G=\SL(2,\bbR)$ from representation of $L=D\cup -D$
using the upper triangular parabolic $LN,$ $N=N_+.$ To do so consider the
action of $L$ on the right on $G/N_+=KD.$ This gives rise to a diagram of
maps
\begin{equation}
\xymatrix{
&G/N\\
&G/N\times L\ar[dl]_C\ar[dr]^R\ar[u]^L\\
G/N&& L
}
\label{CoSL2.257}\end{equation}
where (this may be a bad choice of normalization) the top map is the
left projection, the lower left map is product map , $(kd,l)\longmapsto kdl,$ and
the lower right map is projection \emph{and} inversion, $(kd,l)\longmapsto l^{-1}.$

Now, if we define the compactification of the product to be the b-stretched product
\begin{equation}
(G/N\times L)[1]=K\times [D[1]\times L[1],\{q(e_1),q(e_1)\},\{q(e_2),q(e_2)\}]
\label{CoSL2.258}\end{equation}
obtained by blowing up four of the eight corners, where $L[1]=D[1]\cup-D[1],$ then:-

\begin{lemma}\label{CoSL2.259} The diagram of fibrations \eqref{CoSL2.257}
  extends to a diagram of b-fibrations 
\begin{equation}
\xymatrix{
&G/N[1]\\
&(G/N\times L)[1]\ar[dl]_{C}\ar[dr]^{R}\ar[u]^{L}\\
G/N[1]&& L[1].
}
\label{CoSL2.260}\end{equation}
\end{lemma}

\begin{proof} Basically this is the stretched product for the multiplicative group
  $L$ and ultimately $D.$
\end{proof}

The Harish-Chandra space of $L$ is 
\begin{equation}
\HC(L)=(\ilog\rho )^{\infty}\cA(L[1]).
\label{CoSL2.262}\end{equation}
That is, the bounded conormal functions with log-rapid decay, without weight.

As above, we can deduce a product from \eqref{CoSL2.260} as a bilinear map 
\begin{equation}
\begin{gathered}
\HC(G/N)\times\HC(L)\longrightarrow \HC(G/N),\\
\phi\circ\psi=L_*(C^*\phi\cdot R^*(\delta^{\ha} \psi) dl).
\end{gathered}
\label{CoSL2.261}\end{equation}
Here $\delta :L=\begin{pmatrix}l&0\\0&l^{-1}
\end{pmatrix}\longrightarrow l^2.$ So $\delta (l^{-1})=\delta (l)^{-1}.$ 

\begin{proof}
The compactified space is really two copies of the product $K\times D[2]$
where $D[2]$ is the b-resolution of $D^2,$ so with the two diagonal corners
blown up. I believe $dl$ is the b-differential (so confusingly $dl/l).$ If
I have not messed up here, the factor $\delta ^{\ha}$ shifts the weighting
on the unweighted space $\HC(L)$ so that it looks like the restriction of
$\HC(G/N)$ to $L[1],$ as a p-submanifold of $G/N[1].$ As a result this
should be like the action of $\HC(G)$ on $\HC(G/N)$ below.
\end{proof}

For any Fr\'echet space $V$ there is no problem in defining $\HC(G;V),$
$\HC(G/N;V)$ and so on, just as the subspace of \ci\ maps into $V$ which
satisfy the same estimates as $\HC$ but for each of the seminorms on $V.$
Now, the induced representation corresponding to $\pi$ acts on a Fr\'echet
space which is a closed subspace of $\HC(G/N;V).$ Namely 
\begin{equation}
\HC_{\pi}(G/N;V)=\{u\in\HC(G/N;V);u(kdl)=\delta(l)^{-\ha}\pi(l)^{-1}u(kd)\}.
\label{CoSL2.263}\end{equation}

This is supposed to carry an induced tempered representation of $G.$
Clearly $\HC(G;N)$ itself has a left action of $G$ and this leaves
$\HC_{\pi}(G/N;V)$ invariant.

The first claim is that the condition in \eqref{CoSL2.263} can be expressed
in terms of the maps in \eqref{CoSL2.260}.

If $u\in\HC(G/N;V),$ $\tilde\pi$ is the bilinear map from $\pi$ and
$\psi\in\HC(L)$ then $L^*u:G/N\times L\longrightarrow V$ with corresponding
boundary behaviour on $(G/N\times L)[1].$ Then consider the
composite map corresponding to \eqref{CoSL2.263}
\begin{equation}
\delta (l)^{\ha}\psi(l)\pi(l)u(kdl):G/N\times L\longrightarrow V
\label{CoSL2.264}\end{equation}
Then the condition \eqref{CoSL2.263} should imply, and actually reduce to 
\begin{equation}
\int \delta (l)^{\ha}\psi(l)\pi(l)u(kdl)=\tilde\pi(\psi,u(kd))
\label{CoSL2.266}\end{equation}
and this in turn is written more compactly as
\begin{equation}
L_*(R^*(\delta^{\ha} \psi\pi)C^*u\cdot dl)=\tilde\pi(\psi,u)\Min\HC(G/N;V).
\label{CoSL2.267}\end{equation}
Of course in C+H this is written in terms of quotients of completed tensor products.

\section{Intertwining}\label{S.int}

\begin{figure}
\begin{tikzpicture}
\def\R{60};
\def\n{10}; 
\def\h{7/10*\R};

\coordinate (orc) at (0 pt, 0 pt);
\path [name path=orcircle] (orc) circle [radius=\R pt]; 

\pgfmathsetmacro\cenblowr{2 *\R / (\n+1)};
\pgfmathsetmacro\cenblowy{\h*(1-2/(\n +1))};
\coordinate (cenblowtopc) at (0 pt, \cenblowy pt);
\path [name path=cenblowtop] (cenblowtopc) circle [radius=\cenblowr pt]; 
\coordinate (cenblowbotc) at (0 pt, -\cenblowy pt);
\path [name path=cenblowbot] (cenblowbotc) circle [radius=\cenblowr pt]; 

\pgfmathsetmacro\bigr{\R / 2};
\pgfmathsetmacro\secblowy{\h+\bigr-\R/8};
\coordinate (secblowtopc) at (0 pt, \secblowy pt);
\path [name path=secblowtop]  (secblowtopc) circle [radius=\bigr pt]; 
\coordinate (secblowbotc) at (0 pt, -\secblowy pt);
\path [name path=secblowbot]  (secblowbotc) circle [radius=\bigr pt]; 

\path [name path=diag] (0 pt, -\R pt) -- (0 pt, \R pt); 

\path[name intersections = {of = orcircle and secblowtop, by = {ortopa, ortopb}}];
\pgfmathanglebetweenpoints{\pgfpointanchor{orc}{center}}{\pgfpointanchor{ortopa}{center}}
\edef\angleortopa{\pgfmathresult}
\pgfmathanglebetweenpoints{\pgfpointanchor{orc}{center}}{\pgfpointanchor{ortopb}{center}}
\edef\angleortopb{\pgfmathresult}
\path[name intersections = {of = orcircle and secblowbot, by = {orbota, orbotb}}];
\pgfmathanglebetweenpoints{\pgfpointanchor{orc}{center}}{\pgfpointanchor{orbota}{center}}
\edef\angleorbota{\pgfmathresult}
\pgfmathanglebetweenpoints{\pgfpointanchor{orc}{center}}{\pgfpointanchor{orbotb}{center}}
\edef\angleorbotb{\pgfmathresult}
\pgfmathsetmacro\angleorbotb{\angleorbotb-360};
\draw[blue, thick] (orc)+(\angleortopb:\R pt) arc (\angleortopb:\angleorbota:\R pt);
\draw[blue, thick] (orc)+(\angleorbotb:\R pt) arc (\angleorbotb:\angleortopa:\R pt);

\path[name intersections = {of = secblowtop and cenblowtop, by = {sectopa, sectopb}}];
\pgfmathanglebetweenpoints{\pgfpointanchor{secblowtopc}{center}}{\pgfpointanchor{sectopa}{center}}
\edef\anglesectopa{\pgfmathresult}
\pgfmathanglebetweenpoints{\pgfpointanchor{secblowtopc}{center}}{\pgfpointanchor{sectopb}{center}}
\edef\anglesectopb{\pgfmathresult}
\pgfmathanglebetweenpoints{\pgfpointanchor{secblowtopc}{center}}{\pgfpointanchor{ortopa}{center}}
\edef\anglesectopA{\pgfmathresult}
\pgfmathanglebetweenpoints{\pgfpointanchor{secblowtopc}{center}}{\pgfpointanchor{ortopb}{center}}
\edef\anglesectopB{\pgfmathresult}
\draw[cyan, thick] (secblowtopc)+(\anglesectopb:\bigr pt) arc (\anglesectopb:\anglesectopA:\bigr pt);
\draw[cyan, thick] (secblowtopc)+(\anglesectopB:\bigr pt) arc (\anglesectopB:\anglesectopa:\bigr pt);
\path[name intersections = {of = secblowbot and cenblowbot, by = {secbota, secbotb}}];
\pgfmathanglebetweenpoints{\pgfpointanchor{secblowbotc}{center}}{\pgfpointanchor{secbota}{center}}
\edef\anglesecbota{\pgfmathresult}
\pgfmathanglebetweenpoints{\pgfpointanchor{secblowbotc}{center}}{\pgfpointanchor{secbotb}{center}}
\edef\anglesecbotb{\pgfmathresult}
\pgfmathanglebetweenpoints{\pgfpointanchor{secblowbotc}{center}}{\pgfpointanchor{orbota}{center}}
\edef\anglesecbotA{\pgfmathresult}
\pgfmathanglebetweenpoints{\pgfpointanchor{secblowbotc}{center}}{\pgfpointanchor{orbotb}{center}}
\edef\anglesecbotB{\pgfmathresult}
\draw[cyan, thick] (secblowbotc)+(\anglesecbotb:\bigr pt) arc (\anglesecbotb:\anglesecbotA:\bigr pt);
\draw[cyan, thick] (secblowbotc)+(\anglesecbotB:\bigr pt) arc (\anglesecbotB:\anglesecbota:\bigr pt);

\path[name intersections = {of = secblowtop and cenblowtop, by = {sectopa, sectopb}}];
\pgfmathanglebetweenpoints{\pgfpointanchor{cenblowtopc}{center}}{\pgfpointanchor{sectopa}{center}}
\edef\anglecentopa{\pgfmathresult}
\pgfmathanglebetweenpoints{\pgfpointanchor{cenblowtopc}{center}}{\pgfpointanchor{sectopb}{center}}
\edef\anglecentopb{\pgfmathresult}
\pgfmathsetmacro\anglecentopa{\anglecentopa-360};
\draw[green, thick] (cenblowtopc)+(\anglecentopa:\cenblowr pt) 
		arc (\anglecentopa:\anglecentopb:\cenblowr pt);
\path[name intersections = {of = secblowbot and cenblowbot, by = {secbota, secbotb}}];
\pgfmathanglebetweenpoints{\pgfpointanchor{cenblowbotc}{center}}{\pgfpointanchor{secbota}{center}}
\edef\anglecenbota{\pgfmathresult}
\pgfmathanglebetweenpoints{\pgfpointanchor{cenblowbotc}{center}}{\pgfpointanchor{secbotb}{center}}
\edef\anglecenbotb{\pgfmathresult}
\pgfmathsetmacro\anglecenbota{\anglecenbota-360};
\draw[green, thick] (cenblowbotc)+(\anglecenbota:\cenblowr pt) 
		arc (\anglecenbota:\anglecenbotb:\cenblowr pt);

	\pgfmathsetmacro\ycoord{(\R+\h)/2*(1-9/(\n +1))}
	\pgfmathsetmacro\radius{\R-\ycoord}
	\coordinate (horoc) at (0 pt,\ycoord pt);
	\path [name path=horo] (0 pt,\ycoord pt) circle [radius=\radius pt];
	\path[name intersections = {of = secblowtop and horo, by = {horotopa, horotopb}}];
	\pgfmathanglebetweenpoints{\pgfpointanchor{horoc}{center}}{\pgfpointanchor{horotopa}{center}}
	\edef\anglehorotopa{\pgfmathresult}
	\pgfmathanglebetweenpoints{\pgfpointanchor{horoc}{center}}{\pgfpointanchor{horotopb}{center}}
	\edef\anglehorotopb{\pgfmathresult}
	\pgfmathsetmacro\anglehorotopa{\anglehorotopa-360};

	\pgfmathsetmacro\ycoord{(\R+\h)/2*(1-9/(\n +1))}
	\pgfmathsetmacro\horor{\R-\ycoord}
	\coordinate (horoc) at (0 pt,\ycoord pt);
	\path [name path=horo] (0 pt,\ycoord pt) circle [radius=\horor pt];
	\path[name intersections = {of = secblowtop and horo, by = {horotopa, horotopb}}];
	\pgfmathanglebetweenpoints{\pgfpointanchor{horoc}{center}}{\pgfpointanchor{horotopa}{center}}
	\edef\anglehorotopa{\pgfmathresult}
	\pgfmathanglebetweenpoints{\pgfpointanchor{horoc}{center}}{\pgfpointanchor{horotopb}{center}}
	\edef\anglehorotopb{\pgfmathresult}
	\pgfmathsetmacro\anglehorotopa{\anglehorotopa-360};
	\path[name path=horo] (horoc)+(\anglehorotopa:\horor pt) 
		arc (\anglehorotopa:\anglehorotopb:\horor pt); 
	\pgfmathsetmacro\bigrNear{1.15*\bigr};
	\path [name path=connectingbot] (secblowbotc) circle [radius = \bigrNear pt];
	\pgfmathsetmacro\cenblowrNear{1.4*\cenblowr};
	\path [name path=cenblowbotNear] (cenblowbotc) circle [radius=\cenblowrNear pt]; 
\path[name intersections = {of = horo and connectingbot, by = {NearA, NearB}}];
\pgfmathanglebetweenpoints{\pgfpointanchor{horoc}{center}}{\pgfpointanchor{NearA}{center}}
\edef\angletopNearA{\pgfmathresult}
\pgfmathanglebetweenpoints{\pgfpointanchor{horoc}{center}}{\pgfpointanchor{NearB}{center}}
\edef\angletopNearB{\pgfmathresult}
\pgfmathsetmacro\angletopNearA{\angletopNearA-360};
\pgfmathsetmacro\angletopNearB{\angletopNearB-360};
\path[name intersections = {of = cenblowbotNear and connectingbot, by = {NearC, NearD}}];
\pgfmathanglebetweenpoints{\pgfpointanchor{cenblowbotc}{center}}{\pgfpointanchor{NearC}{center}}
\edef\angletopNearC{\pgfmathresult}
\pgfmathanglebetweenpoints{\pgfpointanchor{cenblowbotc}{center}}{\pgfpointanchor{NearD}{center}}
\edef\angletopNearD{\pgfmathresult}
\pgfmathsetmacro\angletopNearC{\angletopNearC-360};
\pgfmathsetmacro\angletopNearD{\angletopNearD-360};
\pgfmathanglebetweenpoints{\pgfpointanchor{secblowbotc}{center}}{\pgfpointanchor{NearA}{center}}
\edef\angleconnA{\pgfmathresult}
\pgfmathanglebetweenpoints{\pgfpointanchor{secblowbotc}{center}}{\pgfpointanchor{NearB}{center}}
\edef\angleconnB{\pgfmathresult}
\pgfmathanglebetweenpoints{\pgfpointanchor{secblowbotc}{center}}{\pgfpointanchor{NearC}{center}}
\edef\angleconnC{\pgfmathresult}
\pgfmathanglebetweenpoints{\pgfpointanchor{secblowbotc}{center}}{\pgfpointanchor{NearD}{center}}
\edef\angleconnD{\pgfmathresult}

\def\nNear{10}
\foreach \x in {0,1,2,...,\nNear} { 
\pgfmathsetmacro\newangle{\anglehorotopa+(\angletopNearB-\anglehorotopa)*\x/(\nNear+1)};
\path (horoc)+(\newangle:\horor pt)
	coordinate (First\x);}

\def\nNear{5}
\foreach \x in {1,2,...,\nNear} { 
\pgfmathsetmacro\newangle{\angleconnB+(\angleconnD-\angleconnB)*\x/(\nNear+1)};
\path (secblowbotc)+(\newangle:\bigrNear pt)
	coordinate (Second\x);
}

\def\nNear{10}
\foreach \x in {1,2,...,\nNear} { 
\pgfmathsetmacro\newangle{\angletopNearD+(\angletopNearC-\angletopNearD)*\x/(\nNear+1)};
\path (cenblowbotc)+(\newangle:\cenblowrNear pt)
	coordinate (Third\x);
}

\def\nNear{5}
\foreach \x in {1,2,...,\nNear} { 
\pgfmathsetmacro\newangle{\angleconnC+(\angleconnA-\angleconnC)*\x/(\nNear+1)};
\path (secblowbotc)+(\newangle:\bigrNear pt)
	coordinate (Fourth\x);}

\def\nNear{10}
\pgfmathsetmacro\nNearp{\nNear+1}
\foreach \x in {1,2,...,\nNearp} { 
\pgfmathsetmacro\newangle{\angletopNearA+(\anglehorotopb-\angletopNearA)*\x/(\nNear+1)};
\path (horoc)+(\newangle:\horor pt)
	coordinate (Fifth\x);}

\def\pts{
(First1) (First2) (First3) (First4) (First5) (First6) (First7) (First8) (First9) (First10) 
(Second1) (Second2) (Second3) (Second4) (Second5)
(Third1) (Third2) (Third3) (Third4) (Third5) (Third6) (Third7) (Third8) (Third9) (Third10)
(Fourth1) (Fourth2) (Fourth3) (Fourth4) (Fourth5)  
(Fifth1) (Fifth2) (Fifth3) (Fifth4) (Fifth5) (Fifth6) (Fifth7) (Fifth8) (Fifth9) (Fifth10) 
};
\def\mypath{ (First0) to [curve through={\pts}] (Fifth11)};
\draw [red] \mypath;
\draw [orange, transform canvas={rotate=180}] \mypath;

\foreach \x in {3,4,6} {
	\pgfmathsetmacro\ycoord{(\R+\h)/2*(0.99-\x/(\n +1))}
	\pgfmathsetmacro\radius{\R-\ycoord}
	\coordinate (horoc) at (0 pt,\ycoord pt);
	\path [name path=horo] (0 pt,\ycoord pt) circle [radius=\radius pt];
	\path[name intersections = {of = secblowtop and horo, by = {horotopa, horotopb}}];
	\pgfmathanglebetweenpoints{\pgfpointanchor{horoc}{center}}{\pgfpointanchor{horotopa}{center}}
	\edef\anglehorotopa{\pgfmathresult}
	\pgfmathanglebetweenpoints{\pgfpointanchor{horoc}{center}}{\pgfpointanchor{horotopb}{center}}
	\edef\anglehorotopb{\pgfmathresult}
	\pgfmathsetmacro\anglehorotopa{\anglehorotopa-360};
	\draw[red] (horoc)+(\anglehorotopa:\radius pt) 
		arc (\anglehorotopa:\anglehorotopb:\radius pt); }

\foreach \x in {3,4,6} {
	\pgfmathsetmacro\ycoord{-(\R+\h)/2*(1.03-\x/(\n +1))}
	\pgfmathsetmacro\radius{\R+\ycoord}
	\coordinate (horoc) at (0 pt,\ycoord pt);
	\path [name path=horo] (0 pt,\ycoord pt) circle [radius=\radius pt];
	\path[name intersections = {of = secblowbot and horo, by = {horobota, horobotb}}];
	\pgfmathanglebetweenpoints{\pgfpointanchor{horoc}{center}}{\pgfpointanchor{horobota}{center}}
	\edef\anglehorobota{\pgfmathresult}
	\pgfmathanglebetweenpoints{\pgfpointanchor{horoc}{center}}{\pgfpointanchor{horobotb}{center}}
	\edef\anglehorobotb{\pgfmathresult}
	\pgfmathsetmacro\anglehorobota{\anglehorobota-360};
	\draw[orange] (horoc)+(\anglehorobota:\radius pt) 
		arc (\anglehorobota:\anglehorobotb:\radius pt); }

\path[name intersections = {of = diag and cenblowtop, by = {topd', topd}}];
\path[name intersections = {of = diag and cenblowbot, by = {botd, botd'}}];
\draw (botd) -- (topd);

\end{tikzpicture}
\caption{$ G[1;N_\pm]$}\label{Figpm}
\end{figure}
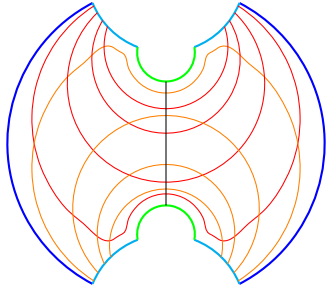

If $N_-$ is the opposite, the transpose, of $N_+=N$ then 
\begin{equation}\label{9.11.I}
	G/N_- \cong KA
\end{equation}
so the two space are naturally identified. However under this
identification the corresponding Harish-Chandra spaces $\HC(G/N_{\pm})$ are
\emph{not} identified. Rather 
\begin{equation}
\begin{gathered}
\HC(G/N_{\pm})=(\ilog\rho )^{\infty}\rho _{0\pm}^\mu\rho_{1\pm}^{-\mu}\cA(KA)\\
\Longrightarrow \HC(G/N_+)=\delta _+\HC(G/N_-),\
\delta _+=\rho _{0+}^{-2\mu}\rho_{1+}^{2\mu}=\rho_{1-}^{-2\mu}\rho_{0-}^{2\mu}=\delta _-^{-1}.
\end{gathered}
\label{CoSL2.244}\end{equation}

The `limiting element' for $N_+,$ $q(e_2),$ is replaced for $N_+$ by
\begin{equation*}
	q(e_1) =
	\begin{pmatrix}
	1 & 0 \\ 0 & 0
	\end{pmatrix}
\end{equation*}
and since this is antipodal we may simultaneously perform the
resolutions for both $N_+$ and for $N_-$ obtaining
\begin{multline}
G[1;N_\pm]=K\times \At[1;N_+,N_-],\\
\At[1;N_+,N_-]=[[\At[1];\{q(e_2)\},\{q(e_1)\}]; \pa_{2,+},\pa_{2,-}]
\end{multline}
since the centres of blow-up are disjoint -- see Figure~\ref{Figpm}.

This space can be used to analyze the well-known intertwining operators
$\cJ_{\pm}$ which using \eqref{9.11.I} can be seen as integral transforms
\begin{equation}\label{9.11.J}
	\cJ_{\pm}:\CIc(KA) \lra \CIc(KA),\
\cJ_{\pm}(u) = (\pi_{N_\pm})_*((\pi_{N_\mp}^*u)\cdot dn_{\pm}).
\end{equation}

\begin{proposition} The fibrations $\pi_{\pm}:G[1;N_{\pm}] \lra
  G/N_{\pm}[1]$ lift to b-fibrations
\begin{equation*}
 \xymatrix{
&G[1;N_\pm]\ar[dr]\ar[dl]\\
G/N_-[1]&&G/N_+[1]
}
\end{equation*}
and the intertwining operators $\cJ_{\pm}$ in \eqref{9.11.J} extend to
continuous linear operators
\begin{equation}
\cJ_{\pm}:\HC(G/N_\mp) \lra \delta^\ha_{\pm}\CI(G/N_{\pm})+\HC(G/N_{\pm})
\label{CoSL2.251}\end{equation}
where the non-trivial leading term is given explicitly as an integral
%
\begin{equation}
\begin{gathered}
K_{\pm}:\HC(G/N_{\mp})\longrightarrow \CI(K),\ K_{\pm}f=\delta^{-\ha}_{0\pm}\cJ_{\pm}f\big|_{\rho_{0\pm}=0},\\
K_{\pm}(f)(\theta)=\int _{A}\delta _{\mp}^{\ha}\left(f(a,\theta+\pi/2)-f(a,\theta-\pi/2)\right)da.
\end{gathered}
\label{CoSL2.247}\end{equation}
\end{proposition}

\begin{remark}\label{CoSL2.248} The computations of Crisp and Higson
  \cite{MR3666051} show that $\cJ_{\mp}$ have continuous right inverses $I_{\pm}$
\begin{equation}
\xymatrix{\HC(G/N_\pm)\ar@{->>}[r]^-{I_\pm}&\Nul(K_{\mp})},\ \cJ_{\mp}I_{\pm}=\Id.
\label{CoSL2.249}\end{equation}

\end{remark}

\begin{proof} The first step is to analyse the pull-back of $\HC(G/N_+)$ to $G[1,N_\pm].$
This factors through the pull-back to $G[1;N_+]$ where the projection to
$G/N_+[1]$ is a fibration, so  
\begin{equation}
\pi_+^*\HC(G/N_+)\subset \iota_2(\ilog\rho _{0+})^{\infty}(\ilog\rho
_{1+})^{\infty}\rho_{0+}^\mu\rho_{1+}^{-\mu}\cA(G[1;N_+]). 
\label{CoSL2.245}\end{equation}
The extra blow-ups in the passage from $G[1,N_+]$ to $G[1,N_\pm]$ occur at
$q(e_1),$ in the interior of the face defined by $\rho_{0+},$ and at the
boundary of the resulting front face. It follows directly that  
\begin{multline}
\pi_+^*\HC(G/N_+)\subset \\
\iota_2(\ilog\rho _{0})^{\infty}(\ilog\rho
_{1+})^{\infty}(\ilog\rho
_{2-})^{\infty}(\ilog\rho
_{1-})^{\infty}\rho_{0}^\mu\rho_{2-}^{2\mu}\rho_{1-}^\mu\rho_{1+}^{-\mu}\cA(G[1;N_\pm]).
\label{CoSL2.246}\end{multline}
where now $\rho _0$ defines the `old boundary' outside the two
blow-ups. As usual, this is not an equality.

Push-forward is relative to the generating vector field $V_-$ for the
action of $N_-.$ From Proposition~\ref{G1N} this is smooth on $G[1,N_-]$
and of the form $\rho _{1-}\rho _{2-}^2W_-$ with $W_-\rho _2\not=0$ at the
boundary but $W_-$ tangent to the other boundaries and in particular non-zero
at $q(e_2).$ Lifted to $G[1;N_+,N_-]$ this becomes singular and of the form
\begin{equation}
V_-=\rho_{1+}^{-1}\rho_{2+}^{-1}\rho _{1-}\rho _{2-}\tilde{W}_-,\ \tilde{W}\in\Vb(G[1;N_\pm])
\label{9.11.2018.4}\end{equation}
where now $\tilde{W_-}$ is a non-vanishing smooth b-vector field which
spans the null space of the b-differential of the stretched projection to
$G/N_-[1].$ Overall then $(J_+f)\nu_{\bo},$ for $f\in\HC(G/N_+)$ is the
image of some some b-density
\begin{multline*}
\iota_{2+}(\ilog\rho_{1+})^\infty(\ilog\rho_0)^\infty(\ilog\rho_{1-})^\infty(\ilog\rho_{2-})^\infty
\rho _{1+}^{\ha}\rho_{2+}\rho_0^\ha\rho_{1-}^{-\ha}g\nu_{\bo},\\
g\in\cA(G[1;N_+,N_-]).
\label{9.11.2018.5}\end{multline*}
under pushforward with respect to $\pi_-.$ This is indeed
integrable across the fixed boundary $\{\rho _{2-}=0\}$ (because of the rapid
log-decay) and by the push-forward theorem therefore lies in 
\begin{equation}
(\ilog\rho
_{0-}^\infty(\ilog\rho _{1-})^\infty\rho
_{0-}^\mu\rho_{1-}^{-\mu}\cA(G/N_-[1])\nu_{\bo},
\label{CoSL2.250}\end{equation}
so \eqref{CoSL2.251} holds.

Continuity also follows from this argument. 
\end{proof}

\appendixbody
\section*{Appendix: Conormal functions}

Since the spaces of log-rapid decay conormal functions are not well-known
we recall here, without proofs, some of the properties of conormal
functions to put these in context.

We start by recalling the case of a compact manifold with boundary, $X.$ If
$\cV(X)$ is the Lie algebra of all smooth vector fields -- meaning smooth
up to the boundary -- then using the action on extendible distributions
(so just in the interior) smooth functions are characterized by 
\begin{equation}
\CI(X)=\{u\in L^\infty(X);\Diff*(X)u\subset L^{\infty}(X)\}.
\label{CoSL2.140}\end{equation}
Here $\Diff*(X)$ is the enveloping algebra of $\cV(X),$ the space of linear
differential operators with coefficients smooth on $X.$ The spaces of order
at most $k$ are finitely spanned over $\CI(X)$ and the Fr\'echet topology
on $\CI(X)$ is given by the corresponding $L^\infty$ norms in
\eqref{CoSL2.140}.

The conormal functions (with respect to $L^\infty,$ these could also very
properly be called `symbols') are defined by direct analogy with
\eqref{CoSL2.140} but replacing $\cV(X)$ by its (more intrinsic)
sub-algebra 
\begin{equation}
\bV(X)=\{V\in\cV(X);V\text{ is tangent to }\pa X\}.
\label{CoSL2.141}\end{equation}
The tangency condition can be restated in terms of a smooth boundary
defining function $\rho \in\CI(X),$ $\{\rho >0\}=X\setminus\pa X,$ $d\rho
\not=0$ on $\pa X.$ Namely if $V \in \cV(X)$ then $V\in\bV(X)$ if and only
if $(V\rho)/\rho \in\CI(X).$ Then $\Diffb*(X)\subset\Diff*(X)$ is the corresponding
enveloping algebra and we define the space of conormal functions by
\begin{equation}
\cA(X)=\{u\in L^\infty(X);\Diffb*(X)u\subset L^{\infty}(X)\}.
\label{CoSL2.142}\end{equation}
This is a Fr\'echet space with the seminorms defined in the same manner and 
\begin{equation}
\CI(X)\subset\cA(X)
\label{CoSL2.143}\end{equation}
with the inclusion continuous.

We can recover this smooth subspace by considering a `radial vector
field'. This is an element $R\in\bV(X),$ usually taken to be real, with the
normalizing condition that 
\begin{equation}
R\rho =\rho +a\rho ^2,\ a\in\CI(X).
\label{CoSL2.144}\end{equation}
In local coordinates in which $\rho =x$ then $R=x\pa_x+xT$ where
$T\in\bV(X)$ locally, and local radial vector fields can be patched to give
a global radial vector field. Having chosen the radial vector field
consider the `test operators'
\begin{equation}
T(R,k)=R(R-1)\dots (R-k)\in\Diffb k(X).
\label{CoSL2.145}\end{equation}
The smooth subspace is characterized by the `Taylor series' conditions 
\begin{equation}
u\in\cA(X),\ T(R,k)u\in\rho ^kL^\infty(X)\ \forall\ k\Longrightarrow u\in\CI(X).
\label{CoSL2.146}\end{equation}

As well as the `bounded conormal functions' defined by \eqref{CoSL2.142} we
need \emph{weighted} versions of such spaces. By a \emph{weight} $0<\alpha
\in\CI(X\setminus\pa X)$ (defined only on the interior of $X)$ we mean 
functions with the iterative property  
\begin{equation}
P\alpha \in \alpha L^\infty(X)\Longleftrightarrow (P\alpha )/\alpha \in
L^\infty(X)\ \forall\ P\in\Diffb*(X).
\label{CoSL2.147}\end{equation}
The most obvious example is a defining function $\rho \in\CI(X).$ Two
weights are equivalent if they are bounded relative to each other 
\begin{equation}
\frac1c\alpha _1\le\alpha _2\le c\alpha _2,\ c>0
\label{CoSL2.163}\end{equation}
and only the behaviour near the boundary is significant.  The weighted
spaces discussed below only depend on the equivalence class of the weight
and any weight is equivalent to one which is a function of a radial
variable, reducing to the one-dimensional case. The only examples which
arise here are powers $x^t$ and $-\log x.$

The product of two weights is also a weight. If $\alpha$ is a weight then
for any $t\in\bbR,$ $\alpha ^t$ is a weight. Significantly in the present
setting if $\inf \alpha>1$ then $\log\alpha$ is also a weight.

For any weight the corresponding weighted conormal space is defined by 
\begin{equation}
\alpha \cA(X)=\{u:X\setminus\pa X\longrightarrow \bbC; (Pu)/\alpha \in
L^{\infty}(X)\ \forall\ P\in\Diffb*(X)\}.
\label{CoSL2.148}\end{equation}
In particular, $\alpha \in\alpha \cA(X)$ and as the notation implicitly indicates
\begin{equation}
u\in \alpha \cA(X)\Longleftrightarrow u/\alpha \in\cA(X)
\label{CoSL2.149}\end{equation}
as a consequence of the estimates \eqref{CoSL2.147}. That is,
multiplication by $\alpha$ is an isomorphism of $\cA(X)$ onto $\alpha
\cA(X).$ For two weights
\begin{equation}
\alpha \le C\beta \Longrightarrow \alpha \cA(X)\subset\beta \cA(X).
\label{CoSL2.150}\end{equation}

If $\alpha$ is a bounded weight and $\beta$ is a weight is convenient to consider $\alpha
^\infty\beta $ as a formal weight in the sense that 
\begin{equation}
\alpha ^\infty\beta \cA(X)=\bigcap_k\alpha ^k\beta \cA(X).
\label{CoSL2.151}\end{equation}
These are again Fr\'echet spaces and if $\alpha$ \emph{vanishes at the boundary}
\begin{equation}
\lim_{\epsilon \downarrow0}\sup_{\rho <\epsilon }\alpha =0\Mthen
\CIc(X\setminus\pa X)\text{ is dense in }\alpha ^\infty\beta \cA(X).
\label{CoSL2.152}\end{equation}

The Harish-Chandra space in the case of $\SL(2,\bbK)$ is $(\ilog\rho
)^\infty\rho ^\kappa \cA(G[1])$ where for $\rho <1,$
\begin{equation*}
	\ilog\rho = \frac1{\log\frac1\rho}
\end{equation*}
is a boundary defining
function so in particular this density statement applies.

For $\SL(n,\bbK)$ and even for $\SL(2,\bbK)$ when we consider $G[2]$ and
related compactifications, we need to consider conormal functions on
compact manifolds with corners. Recall that such a manifold, still denoted
$X,$ is locally modelled on $[0,\infty)^n$ instead of $\bbR^n$ and we
  impose the additional requirement that boundary hypersurfaces -- the
  closures of the components of the subsets of points at which the local
  model is $[0,\infty)\times\bbR^{n-1}$ -- are embedded. This is equivalent
    to requiring that each such boundary hypersurface $H$ has a boundary
    defining function $\rho _H\in\CI(X)$ in the sense completely analogous
    to the boundary case
\begin{equation}
\{\rho _H>0\}=X\setminus H,\ d\rho _H\not=0\Mat H.
\label{CoSL2.153}\end{equation}

It follows that each of the boundary hypersurfaces has a neighbourhood in
 $X$ diffeomorphic to $H\times[0,\epsilon )_\rho.$ This allows all the
  statements above to be generalized rather directly. Namely $\bV(X)$ is
  the Lie algebra of smooth vector fields tangent to all boundary
  hypersurfaces (and hence to all boundary faces). The definition of the
  bounded-conormal space and weights is then formally the same as
  \eqref{CoSL2.142}, \eqref{CoSL2.147} and \eqref{CoSL2.148}.
There are intermediate Lie algebras between $\bV(X)$ and $\cV(X),$ in
particular if $H$ is a hypersurface then 
\begin{equation}
\cV_H(X)=\{V\in\cV(X);V\rho _H\in\rho _H\CI(X)\}
\label{CoSL2.155}\end{equation}
consists of the vector fields which are tangent to $H.$ A weight at $H$ is
then defined by the condition 
\begin{equation}
0<\alpha \in\CI(X\setminus H),\ \Diff*_H\alpha \subset\alpha L^\infty(X)
\label{CoSL2.156}\end{equation}
which implies that $\alpha$ is a weight on $X$ but is also smooth, so
trivial as a weight, up to hypersurfaces other than $H.$ Then if $\cM_1(X)$
is the set of boundary 
hypersurfaces and $\alpha _H$ is a weight for each $H\in\cM_1(X)$ then
taking $\alpha _*$ to be the products of these weights there are
corresponding conormal spaces
\begin{equation}
\alpha _*\cA(X)=\{u\in L^{\infty}(X);u/\Pi_{H\in\cM_1(X)}\alpha _H\in \cA(X)\}.
\label{CoSL2.157}\end{equation}
The properties listed above carry over in a rather direct way and in
particular formal weighted spaces, corresponding to $\alpha _H^\infty\beta
_H$ at any combination of hypersurfaces, are defined if the weight $\alpha
_H$ vanish at $H$ in the sense corresponding to \eqref{CoSL2.152}.

We also use hybrid \ci-conormal spaces; that these make good sense is a
consequence of the local product decomposition near a boundary
hypersurface. We define another formal weight at each boundary
hypersurfaces, $\iota_H.$ If $\alpha _*$ is a collection of weights one of
which is $\iota_H$, then let $\hat\alpha_*(k)$ be the weights where $\iota_H$ is
replaced by $\rho ^k.$ This allows us to define
\begin{equation}
\alpha _*\cA(X)=\{u\in \hat\alpha(0)\cA(X);T(R_H,k)u\in\hat\alpha
(k)\cA(X)\ \forall\ k\}. 
\label{CoSL2.158}\end{equation}
In a local product decomposition this corresponds to smoothness in the
normal variable with values in the conormal space for $H$ where the formal
smooth `weight' is deleted.

One can take the formal smooth weight at any collection of hypersurfaces
and and in particular if one takes this weight at \emph{all} boundary
hypersurfaces then one recovers $\CI(X).$

The conormal spaces have interpolation properties corresponding to
multiplicative properties of the weights. For instance if $w_1(H)$ and
$w_2(H')$ are vanishing weights at different, but possibly
intersecting, hypersurfaces then
\begin{equation}
w_1\cA(X)\cap w_2\cA(X)\subset w_1^{\ha}w_2^{\ha}\cA(X).
\label{CoSL2.165}\end{equation}

From this point onwards we will only consider the special weights given by
the defining functions $\rho _H$ themselves, the formal smoothing weight
and weights related to $\ilog\rho _H.$ 

As remarked above these conormal spaces are analogues of $\CI(X)$ on a
compact manifold without boundary. In the case of compact manifolds with
corners many of the standard functorial results carry over to the smooth
spaces. In particular if $f:M\longrightarrow N$ is a smooth map between
compact manifolds with corners then 
\begin{equation}
f^*:\CI(N)\longrightarrow \CI(M).
\label{CoSL2.159}\end{equation}
For push-forward a stronger condition is needed, that $f$ be a
submersion, 
\begin{equation}
\begin{gathered}
f_*:T_pM\longrightarrow T_{f(p)}N\text{ surjective }\forall\ p\in M\Longrightarrow \\
f_*:\CI(M;\Omega)\longrightarrow \CI(N;\Omega )
\end{gathered}
\label{CoSL2.161}\end{equation}
where it is only natural to push forward densities. 

General smooth maps are not particularly natural in the context of
manifolds with corners -- in general there need be little relationship to
the boundary. So for instance under pull-back, \eqref{CoSL2.159}, vanishing
of $u\in\CI(N)$ at a boundary hypersurface does not have direct
implications for the vanishing of $f^*u$ at boundary hypersurfaces.

It is more natural to work in the category of b-maps -- and these are
indeed the maps that are typically encountered. Here we only consider
\emph{interior} b-maps (meaning the image meets the interior) but drop the
qualifier. A b-map is a smooth map $f:M\longrightarrow N$ with the
additional property that the defining functions pull back appropriately
\begin{equation}
f^*\rho _{H'}=a_{H'}\prod_{H\in\cM_1(M)}\rho _H^{\mu(H,H')},\ \forall\ H'\in\cM_1(N),\
0<a_H\in\CI(M).
\label{CoSL2.160}\end{equation}
The powers $\mu(H,H')$ are necessarily non-negative integers but can
all vanish. 

For such a b-map an analogue of \eqref{CoSL2.159} holds for the conormal
spaces. Namely if $w$ is a weight on $N$ then $f^{\#}w$ is the weight on
$M$ given by $f^*w$ with the addition of the formal smoothing weights at
all hypersurfaces $H\in\cM_1(M)$ for which $\mu(H,H')=0$ for all $H',$ then
\begin{equation}
f^*:\alpha \cA(N)\longrightarrow (f^{\#}\alpha)\cA(M).
\label{CoSL2.166}\end{equation}
For power weights this corresponds to composition in the indices
\begin{equation}
f^*:\rho ^{\kappa'}\cA(N)\longrightarrow
\rho^{\kappa}\cA(M),\ \kappa(H)=\sum\limits_{H'\in\cM_1(N)}\mu(H,H')\kappa'(H').
\label{CoSL2.162}\end{equation}

In general the index $\mu(H_i,H')$ can be non-zero for more than one $H_i\in\cM_1(M)$
and the same $H'\in\cM_1(N).$ If this does not happen, so for each
$H\in\cM_1(M)$ there is at most one $H'$ such that $\mu(H,H')\not=0,$ the
b-map is said to be b-normal -- this corresponds to the absence of boundary
hypersurfaces in $M$ which are mapped into corners of codimension two (or
higher) in $N.$

For the logarithmic weights the pull-back $f^*\ilog \rho _{H'}$ is not a
product of weights. However it is bounded between such products: 
\begin{multline}
\frac1c\prod_{H_i\in\cM_1(M);\mu(H_i,H')\not=0}(\ilog(\rho _{H_i}))^{1/p}\\
\le
f^*i\ilog\rho _H
\le c\prod_{H_i\in\cM_1(M);\mu(H_i,H')\not=0}\ilog(\rho _{H_i}).
\label{CoSL2.167}\end{multline}
Here $p$ is the number of hypersurfaces in the preimage of $H'$ but can be
improved to the maximal number of mutually intersecting hypersurfaces in
the preimage.

For push-forward it is necessary to make stronger assumptions on $f,$ but
weaker than the assumption of a fibration as is needed for
\eqref{CoSL2.161}. Namely it suffices to take $f$ to be a b-fibration. This
corresponds to the three conditions that $f$ be a b-map, that further it
satisfies the b-normal condition, and finally that the b-differential be
surjective. This latter condition can be stated infinitesimally or globally
as the condition that every element $V\in\bV(N)$ is $f$-related to an
element $W\in\bV(M),$  
\begin{equation*}
Wf^*u=f^*(Vu)\ \forall\ u\in\CI(N).
\label{CoSL2.168}\end{equation*}

For a b-fibration there is an analogue of \eqref{CoSL2.161} under an
integrability assumption on the domain. Consider the `fixed' hypersurfaces
on $M,$ those which are not mapped by $f$ into the boundary of $N.$ These
are precisely the hypersurfaces such that $\mu(H,*)=0.$ 
Then a suitable `integral' weight is  
\begin{equation}
I_f=\prod_{H\in\cM_1(M);\mu(H,*)=0}(\ilog \rho _H)^2
\label{CoSL2.170}\end{equation}
where any power greater than one suffices. We also define a weight on $N$
corresponding to the number, $p(H'),$ of boundary hypersurfaces in $M$
mapped into $H.$ As in \eqref{CoSL2.167} this can be refined to the maximal
number of mutually intersecting hypersurfaces in the preimage of $H'.$ Then
\begin{equation}
J_f=\prod_{H'\in\cM_1(N)}(\ilog \rho _{H'})^{1-p(H')}.
\label{CoSL2.171}\end{equation}
Then for any weight $\alpha'$ on $N,$
\begin{equation}
f_*:(f^*w)I_f\cA(M;\Omega _{\bo})\longrightarrow wJ_f\cA(N;\Omega _{\bo})).
\label{CoSL2.169}\end{equation}
Thus there is in general `logarithmic growth' of the push-forward. Note
that the case of a fibration corresponds to $p(H')=1$ and hence no such
factors appear. In \cite{MR93i:58148} the existence of expansions for
push-forward of an integrable function with expansions is discussed --
there may indeed be additional logarithmic terms.

\providecommand{\bysame}{\leavevmode\hbox to3em{\hrulefill}\thinspace}
\providecommand{\MR}{\relax\ifhmode\unskip\space\fi MR }
\providecommand{\MRhref}[2]{%
  \href{http://www.ams.org/mathscinet-getitem?mr=#1}{#2}
}
\providecommand{\href}[2]{#2}

\end{document}